\documentclass[a4paper, 11pt]{amsart}
\usepackage{mathptmx,amssymb,amscd,latexsym, eulervm}
\usepackage{amsmath}
\usepackage{amsthm}
\usepackage{mathdots}
\usepackage[colorlinks=true,citecolor=violet,linkcolor=blue,urlcolor=blue]{hyperref}
\usepackage[dvipsnames]{xcolor}
\usepackage[onehalfspacing]{setspace}
\usepackage{tabularx}
\usepackage{amsfonts}
\usepackage{paralist}
\usepackage{aliascnt}
\usepackage{amscd}
\usepackage{blkarray}
\usepackage{mathbbol}
\usepackage{setspace}
\usepackage{needspace}
\usepackage[inner=2.4cm,outer=2.4cm, bottom=3.2cm]{geometry}
\usepackage{tikz, tikz-cd}
\usepackage{calligra,mathrsfs}

\usepackage{tikz}
\usetikzlibrary{matrix}
\usetikzlibrary{arrows,calc}
\allowdisplaybreaks

\AtBeginDocument{%
	\def\MR#1{}
}

\makeatletter
\@namedef{subjclassname@2020}{%
	\textup{2020} Mathematics Subject Classification}
\makeatother

\newcommand{\RR}{\mathbb{R}}
\newcommand{\kk}{\mathbb{k}}

\newcommand{\fS}{\mathfrak{S}}

\newcommand{\NN}{\normalfont\mathbb{N}}
\newcommand{\ZZ}{{\normalfont\mathbb{Z}}}
\newcommand{\CC}{{\normalfont\mathbb{C}}}

\newcommand{\PP}{{\normalfont\mathbb{P}}}

\newcommand{\xx}{\normalfont\mathbf{x}}

\newcommand{\dd}{\normalfont\mathbf{d}}

\newcommand{\mm}{{\normalfont\mathfrak{m}}}

\newcommand{\QQ}{\mathbb{Q}}
\newcommand{\pp}{\mathfrak{p}}
\newcommand{\aaa}{\mathfrak{a}}

\newcommand{\ttt}{{\normalfont\textbf{t}}}
\newcommand{\sss}{{\normalfont\textbf{s}}}

\newcommand{\rank}{\normalfont\text{rank }}

\newcommand{\Ext}{\normalfont\text{Ext}}

\newcommand{\vol}{{\rm vol}}
\newcommand{\Vol}{{\rm Vol}}
\newcommand{\bn}{{\normalfont\mathbf{n}}}
\newcommand{\bm}{{\normalfont\mathbf{m}}}
\newcommand{\bw}{{\normalfont\mathbf{w}}}
\newcommand{\ee}{{\normalfont\mathbf{e}}}

\newcommand{\codim}{{\rm codim}}

\newcommand{\Ann}{\normalfont\text{Ann}}

\newcommand{\Hom}{{\normalfont\text{Hom}}}

\newcommand{\OO}{\mathscr{O}}

\newcommand{\HL}{\normalfont\text{H}_{\mm}}
\newcommand{\HH}{\normalfont\text{H}}

\newcommand{\AAA}{\mathbb{A}}

\newcommand{\init}{{\rm in}}

\newcommand{\Spec}{{\normalfont\text{Spec}}}

\newcommand{\Cone}{{\normalfont\text{Cone}}}

\newcommand{\Pic}{{\normalfont\text{Pic}}}

\newcommand{\fR}{\mathfrak{R}}

\newcommand{\Gr}{{\rm Gr}}
\newcommand{\Fl}{{{\mathcal{F}\ell}_n}}

\newcommand{\cC}{\mathcal{C}}
\newcommand{\Mat}{{\rm Mat}}
\newcommand{\supp}{\mathrm{supp}}

\newcommand{\Hilb}{\normalfont\text{Hilb}}
\newcommand{\GL}{\normalfont\text{GL}}

\newtheorem{theorem}{Theorem}[section]

\newtheorem{headthm}{Theorem}

\newaliascnt{headcor}{headthm}
\newtheorem{headcor}[headcor]{Corollary}
\aliascntresetthe{headcor}

\newaliascnt{headconj}{headthm}

\aliascntresetthe{headconj}

\newaliascnt{corollary}{theorem}
\newtheorem{corollary}[corollary]{Corollary}
\aliascntresetthe{corollary}

\newaliascnt{claim}{theorem}

\aliascntresetthe{claim}

\newaliascnt{lemma}{theorem}
\newtheorem{lemma}[lemma]{Lemma}
\aliascntresetthe{lemma}

\newaliascnt{conjecture}{theorem}

\aliascntresetthe{conjecture}

\newaliascnt{proposition}{theorem}
\newtheorem{proposition}[proposition]{Proposition}
\aliascntresetthe{proposition}

\theoremstyle{definition}
\newaliascnt{definition}{theorem}
\newtheorem{definition}[definition]{Definition}
\aliascntresetthe{definition}

\newaliascnt{notation}{theorem}

\aliascntresetthe{notation}

\newaliascnt{example}{theorem}
\newtheorem{example}[example]{Example}
\aliascntresetthe{example}

\newaliascnt{examples}{theorem}

\aliascntresetthe{examples}

\newaliascnt{remark}{theorem}
\newtheorem{remark}[remark]{Remark}
\aliascntresetthe{remark}

\newaliascnt{question}{theorem}
\newtheorem{question}[question]{Question}
\aliascntresetthe{question}

\newaliascnt{questions}{theorem}

\aliascntresetthe{questions}

\newaliascnt{problem}{theorem}

\aliascntresetthe{problem}

\newaliascnt{construction}{theorem}

\aliascntresetthe{construction}

\newaliascnt{setup}{theorem}
\newtheorem{setup}[setup]{Setup}
\aliascntresetthe{setup}

\newaliascnt{algorithm}{theorem}

\aliascntresetthe{algorithm}

\newaliascnt{observation}{theorem}

\aliascntresetthe{observation}

\newaliascnt{defprop}{theorem}

\aliascntresetthe{defprop}

\DeclareFontFamily{OT1}{pzc}{}
\DeclareFontShape{OT1}{pzc}{m}{it}{<-> s * [1.100] pzcmi7t}{}
\DeclareMathAlphabet{\mathchanc}{OT1}{pzc}{m}{it}

\def\equationautorefname~#1\null{(#1)\null}
\def\sectionautorefname~#1\null{Section #1\null}
\def\subsectionautorefname~#1\null{\S #1\null}

\def\surjects{\twoheadrightarrow}

\begin{document}
	
	\title{Log-concavity of polynomials arising from equivariant cohomology}

	\author{Yairon Cid-Ruiz}
	\address{Department of Mathematics, North Carolina State University, Raleigh, NC 27695, USA}
	\email{ycidrui@ncsu.edu}
	
	\thanks{Yairon Cid-Ruiz received support from NSF Grant DMS-2502321 and Simons Foundation Travel Support for Mathematicians Award MPS-TSM-00013551.}
	
	\author{Yupeng Li}
	\address{Department of Mathematics, Duke University, Durham, NC 27708, USA}
	\email{ypli@math.duke.edu}
	
	\author{Jacob P. Matherne}
	\address{Department of Mathematics, North Carolina State University, Raleigh, NC 27695, USA}
	\email{jpmather@ncsu.edu}
	
	\date{\today}
	\keywords{equivariant cohomology, multidegrees, cohomology rings, Richardson polynomials, Lorentzian polynomials, covolume polynomials, log-concavity, Schubert polynomials, Richardson varieties, Schubert varieties}
	
	\subjclass[2020]{14M15, 14C15, 14C17, 13H15, 52B40}

	\thanks{Jacob Matherne received support from NSF Grant DMS-2452179 and a Simons Foundation Travel Support for Mathematicians Award MPS-TSM-00007970.}

	\begin{abstract}
		We study the equivariant cohomology classes of torus-equivariant subvarieties of the space of matrices.  
		For a large class of torus actions, we prove that the polynomials representing these classes (up to suitably changing signs) are covolume polynomials in the sense of Aluffi.  
		We study the cohomology rings of complex varieties in terms of Macaulay inverse systems over $\ZZ$.  
		As applications, we show that under certain conditions, the Macaulay dual generator is a denormalized Lorentzian polynomial in the sense of Br\"and\'en and Huh, and we give a characteristic-free extension (over $\mathbb{Z}$) of the  result of Khovanskii and Pukhlikov describing the cohomology ring of toric varieties in terms of volume polynomials.
	\end{abstract}	
	
	\maketitle

	\section{Introduction}
	
	A sequence of real numbers $a_0,a_1,\ldots,a_n$ is called \emph{log-concave} if $a_i^2 \geq a_{i-1}a_{i+1}$ for all $1 \le i \le n-1$.  
	Log-concave sequences naturally appear throughout algebra, combinatorics, and geometry; for thorough references on log-concavity and related concepts, we suggest \cite{B89,S89,B15}.  
	Recently, the theory of \emph{Lorentzian polynomials} was introduced by Br\"and\'en and Huh \cite{LORENTZIAN} (and in \cite{agv21,algv24,algv24b} under the name of \emph{completely log-concave polynomials}) and they have been instrumental in proving log-concavity and related results throughout mathematics, see, e.g.,  \cite{LORENTZIAN,EH20,MS21,HMMSD,BL23,BLP23,R23,algv24b, BES24,HMV24,KMS24,MMS24,RU24}. 
	
	The prototypical examples of Lorentzian polynomials are the volume polynomials of projective varieties \cite{LORENTZIAN}. Likewise, the \emph{covolume polynomials} of Aluffi \cite{ALUFFI_COVOL} are the prototypical examples of the \emph{dually Lorentzian polynomials} of Ross, S\"u\ss, and Wannerer \cite{RSW23}.  We introduce a new family of polynomials that specialize to a number of important polynomials in algebraic combinatorics, and we prove that they are covolume polynomials.

	For two permutations $u,w \in S_n$ with $w \ge u$ in Bruhat order, we define the \emph{double Richardson polynomial} as the product of double Schubert polynomials
	\[
	\fR_{w/u}(\ttt,\sss) \;=\; \fS_u(\ttt,\sss) \fS_{w_0 w}(\ttt,\sss'),
	\]
	where $\sss' = (s_n,\ldots,s_1)$ denotes the reverse of $\sss = (s_1,\ldots,s_n)$. 
	The double Richardson polynomial represents the torus-equivariant class of Richardson varieties in the torus-equivariant cohomology ring of the full flag variety $\Fl$ (see \autoref{rem_Richardson}). 
	Double Richardson polynomials specialize to (double) Schubert polynomials:
	\[
	\fS_u(\ttt,\sss) \;=\; \fR_{w_0/u}(\ttt,\sss)  \qquad \text{and} \qquad \fS_u(\ttt,\mathbf{0}) \;=\; \fS_u(\ttt) \;=\; \fR_{w_0/u}(\ttt, \mathbf{0}).
	\]
	The \emph{Richardson polynomial} $\fR_{w/u}(\ttt) = \fR_{w/u}(\ttt, \mathbf{0})$ is closely related to the \emph{skew Schubert polynomial} $\fS_{w/u}(\ttt)$ of Lenart and Sottile \cite{SKEWSCHUBERT} since both polynomials yield the same class in the cohomology ring $\HH^\bullet(\Fl)$ of the flag variety $\Fl$. 
	But we point out that $\fR_{w/u}(\ttt)$ and $\fS_{w/u}(\ttt)$ may not be equal in general (see \autoref{rem_skew_Schubert}).

	\begin{headthm}[{\autoref{t:equiv_coho_richardson}}]
		\label{thm:listofcovolumepolys}
		The {\rm(}sign-changed{\rm)} double Richardson polynomial $\fR_{w/u}(\ttt,-\sss)$ is a covolume polynomial. 
	\end{headthm}

	Dually Lorentzian polynomials enjoy two nice combinatorial properties (see \autoref{prop_disc_log_conc}): their supports are {\rm M}-convex, and they are discretely log-concave.  A homogeneous polynomial $h = \sum_\bn a_\bn \ttt^\bn$ of degree $d$ with nonnegative coefficients is said to have M-convex support if $\supp(h)$ is the set of integer points of an integral generalized permutohedron in the sense of \cite{Postnikov2009}, and it is said to be discretely log-concave if $a_\bn^2 \geq a_{\bn + \ee_i - \ee_j}a_{\bn - \ee_i + \ee_j}$ for all $\bn \in \NN^n$ and all $1 \le i,j \le n$.
	
	\begin{headcor}[\autoref{cor_polynomials_properties}]
		\label{cor:mconvexanddlc}
		The following polynomials have {\rm M}-convex support and are discretely log-concave:
		\begin{enumerate}[\rm (i)]
			\item {\rm(}sign-changed{\rm)} Double Richardson polynomials $\fR_{w/u}(\ttt,-\sss)$.
			\item Richardson polynomials $\fR_{w/u}(\ttt)$.
			\item {\rm(}sign-changed{\rm)} Double Schubert polynomials $\fS_u(\ttt,-\sss)$.
			\item Schubert polynomials $\fS_u(\ttt)$.
		\end{enumerate}
	\end{headcor}
	
	We note that the M-convexity in \autoref{cor:mconvexanddlc}(iii) recovers a result of \cite{DOUBLE_SCHUBERT_POLYM} and both the M-convexity and the discrete log-concavity in \autoref{cor:mconvexanddlc}(iv) recover results of \cite{HMMSD}.  
	To the best of our knowledge, the discrete log-concavity of double Schubert polynomials is new.
	
	The proof of \autoref{thm:listofcovolumepolys} follows from a general theorem that we prove in \autoref{s:Equiv_coho_multigraded}.  
	More precisely, we prove the following result regarding the equivariant cohomology classes of torus-equivariant subvarieties of the space of matrices.

	\begin{headthm}[{\autoref{thm_equiv}, \autoref{c:cor_equiv_matrices}}]\label{t:equiv_cohomolgy=multidegree_covolume}
		Let ${\rm Mat}_{m,n} = \CC^{m \times n}$ be the space of $m \times n$ matrices with complex entries and consider the natural action of the torus $T = (\CC^*)^m \times (\CC^*)^n$ given by $(g,h) \cdot M = g \cdot M \cdot h^{-1}$ for all $M \in {\rm Mat}_{m,n}$ and $(g,h) \in T$.
		Let $X \subset {\rm Mat}_{m,n}$ be an irreducible $T$-subvariety and $C_X(t_1,\ldots,t_m,s_1,\ldots,s_n)$ be the polynomial representing the class $[X]^T$ of $X$ in $\HH_T^\bullet\left({\rm Mat}_{m,n}\right) = \ZZ[t_1,\ldots,t_m,s_1,\ldots,s_n]$.
		Then we have that $C_X(t_1,\ldots,t_m,-s_1,\ldots,-s_n)$ is a covolume polynomial.
	\end{headthm}
	
	We point out that the result of \autoref{t:equiv_cohomolgy=multidegree_covolume} is sharp in the sense that for arbitrary torus actions we can find simple instances where the equivariant class of an irreducible $T$-variety is represented by a polynomial which does not even have $M$-convex support (see \autoref{examp_bad_grading}). 
	
	\medskip
	
	We also connect the theory of \emph{Macaulay inverse systems} over $\ZZ$ to the theory of Lorentzian polynomials.
	More precisely, we study the cohomology rings of smooth complete irreducible complex algebraic varieties over $\ZZ$ in terms of a suitable \emph{generic version} of the Macaulay inverse system. 
	In this characteristic-free situation over $\ZZ$, we show that the respective Macaulay dual generator is a denormalized Lorentzian polynomial (see \autoref{thm:mcinvsys} below). 
	
	Over a field of characteristic zero, say the field of rational numbers $\QQ$, the even cohomology ring $\HH^{2\bullet}(X, \QQ)$ of a smooth complete complex algebraic variety $X$ is a classical example of an Artinian graded Gorenstein algebra (due to the Poincar\'e duality of $\HH^\bullet(X, \QQ)$).
	Therefore, the Macaulay inverse system stands as a powerful tool to study the even cohomology ring $\HH^{2\bullet}(X, \QQ)$ because Artinian graded Gorenstein algebras can be defined in terms of the annihilator of a single polynomial in a dual polynomial ring (see, e.g., \cite[\S 21.2]{EISEN_COMM}, \cite[\S 13.4]{Cox2011}).
	Perhaps the most spectacular application of this idea is the result of  Khovanskii and Pukhlikov \cite{KP} describing the cohomology ring of a smooth complete toric variety in terms of the annihilator of the volume polynomial. 
	For more recent applications and related results, the reader is referred to \cite{Kaveh_cohom, Lefschetz_book, Leonid4,Leonid3, Leonid2,Alexandra_HL,Leonid1}.

	Our interest is to study the even cohomology ring $\HH^{2\bullet}(X, \ZZ)$ of a smooth complete irreducible complex algebraic variety $X$ as a $\ZZ$-algebra and to find suitable extensions of the aforementioned known results for $\HH^{2\bullet}(X, \QQ)$.
	To substitute the fact that $\HH^{2\bullet}(X, \QQ)$ is an Artinian graded Gorenstein algebra, we should have that the structure morphism $f\colon \Spec\left(\HH^{2\bullet}(X, \ZZ)\right) \rightarrow \Spec(\ZZ)$ is a finite Gorenstein morphism (i.e., $f$ is a finite flat morphism and all its fibers are Gorenstein rings); if this happens, we say $\HH^{2\bullet}(X, \ZZ)$ is Artinian Gorenstein over $\ZZ$.
	We also need to consider generic versions of the Matlis and Macaulay dual functors. (These generic versions have become of interest recently, see, e.g., \cite{KSmith,SPECIALIZATION_MARC_ARON,kleiman2022macaulay,FIB_FULL}.)
	Let $S = \ZZ[x_1,\ldots,x_n]$ be a positively graded polynomial ring and consider the inverse polynomial ring $T = \ZZ[y_1,\ldots,y_n]$ under the identification $y_i=x_i^{-1}$.
	For any homogeneous ideal $I \subset S$ such that the quotient $R = S/I$ is a \emph{finitely generated flat $\ZZ$-module}, we have a successful \emph{Macaulay inverse system} $I^{\perp_{\ZZ}} = \{G \in T \mid g \cdot G = 0 \text{ for all } g \in I\}$ that can be computed with the \emph{graded Matlis dual} $R^{\vee_\ZZ}={}^*\Hom_\ZZ(R, \ZZ)$. 
	We make the necessary developments of these notions in \autoref{subsect_Gor}.
	Our main result regarding cohomology rings is the following theorem.
	
	\begin{headthm}[{\autoref{thm_cohom_Macaulay}, \autoref{cor_toric_setting}}]\label{thm:mcinvsys}
		Let $X$ be a $d$-dimensional smooth complete irreducible complex algebraic variety. 
		Suppose that the even cohomology ring $R = \bigoplus_{i=0}^d\HH^{2i}(X, \ZZ)$ is a flat $\ZZ$-algebra {\rm(}i.e., it is $\ZZ$-torsion-free{\rm)}.
		Let $\rho\colon R_d = \HH^{2d}(X, \ZZ) \rightarrow \ZZ$ be the natural degree map.
		Choose a graded presentation $R \cong  S/I$ where $S = \ZZ[x_1,\ldots,x_n]$, $\delta_i = \deg(x_i) > 0$, and $I \subset S$ is a homogeneous ideal.
		Let $\delta = \delta_1+\cdots+\delta_n$.
		Then the following statements hold: 
		\begin{enumerate}[\rm (i)]
			\item $R$ is Artinian Gorenstein over $\ZZ$.
			\item We have the isomorphisms $\omega_{R/\ZZ} = \Ext_S^n(R, S(-\delta)) \cong {}^*\Hom_\ZZ(R, \ZZ) \cong R(d)$.
			\item Consider the inverse polynomial ring $T = \ZZ[y_1,\ldots,y_n]$ with the identification $y_i = x_i^{-1}$.
			Then the ideal $I \subset S$ is given as the annihilator
			$$
			I \;=\; \lbrace g \in S \mid g \cdot G_R = 0 \rbrace
			$$			
			of the inverse polynomial 
			$$
			G_R(y_1,\ldots,y_n) \;=\; \sum_{\alpha_1\delta_1 + \cdots + \alpha_n\delta_n = d} \rho\big(x_1^{\alpha_1}\cdots x_n^{\alpha_n}\big) \, y_1^{\alpha_1}\cdots y_n^{\alpha_n}  \;\;\in\;\; T = \ZZ[y_1,\ldots,y_n].
			$$		
			\item 
			Assume that each $x_i$ is equal to the first Chern class $c_1(L_i)$ of a nef line bundle $L_i$ on $X$.  
			Then the normalization 
			$$
			N(G_R) \;\in\; \RR[y_1,\ldots,y_n]
			$$ 
			of $G_R$ is a Lorentzian polynomial. 
			\item Assume that $X = X_\Sigma$ is a complete toric variety and that  each $x_i$ is equal to the first Chern class $c_1(L_i)$ of a nef line bundle $L_i$ on $X_\Sigma$.  
			Let $D_i$ be a torus-invariant Cartier divisor on $X_\Sigma$ such that $L_i \cong \OO_{X_\Sigma}(D_i)$.
			Let $P_i = P_{D_i}$ be the polytope associated to $D_i$.  Then 
			\[
			G_R(y_1,\ldots,y_n) \;=\; \sum_{\alpha_1 + \cdots + \alpha_n = d} \mathrm{MV}_\alpha(P_1,\ldots,P_n) \, y_1^{\alpha_1}\cdots y_n^{\alpha_n} \;\;\in\;\; T = \ZZ[y_1,\ldots,y_n],
			\]
			where $\mathrm{MV}_\alpha(P_1,\ldots,P_n)$ is the mixed volume of $P_1,\ldots,P_n$ of type $\alpha = (\alpha_1,\ldots,\alpha_n)$.
		\end{enumerate}
	\end{headthm}

	Motivated by \autoref{thm:mcinvsys}, we say that $G_R \in T = \ZZ[y_1,\ldots,y_n]$ is the \emph{Macaulay dual generator} of the even cohomology ring $R = \HH^{2\bullet}(X, \ZZ)$ over $\ZZ$.
	We perform several explicit computations with examples in \autoref{subset_cohom_ring}.
	
	\bigskip
	
	\noindent
	\textbf{Outline.}
	The basic outline of this paper is as follows. 									In \autoref{sect_prelim}, we recall some basic results and we fix the notation.	
	In \autoref{sec_pos_gradings}, we study the multidegree polynomial of prime ideals in positive but not necessarily standard gradings. 
	We prove \autoref{t:equiv_cohomolgy=multidegree_covolume} in \autoref{s:Equiv_coho_multigraded}.
	The proofs of \autoref{thm:listofcovolumepolys} and \autoref{cor:mconvexanddlc} are given in \autoref{sect_equiv_cohom_Richardson}.
	Finally, \autoref{sect_cohom_rings} contains the proof of \autoref{thm:mcinvsys}.
	
	\section{Preliminaries}
	\label{sect_prelim}
	
	In this section, we set up the notation used throughout the paper. We also present some preliminary results necessary to prove our main theorems.

	\subsection{Double Schubert polynomials}
	
	Let $p \ge 1$ be a positive integer.
	Let $[p] := \{1,2,\dots, p\}$ and $\binom{[p]}{k}$ denote the set of size $k$ subsets of $[p]$. 
	For $I= \{i_1<i_2<\cdots<i_k\}$, $J = \{j_1<j_2<\cdots<j_k\}\in \binom{[p]}{k}$, we define a partial order on $\binom{[p]}{k}$ where $I\leq J$ if $i_m\leq j_m$ for all $m\in [k]$. 
	We write $S_p$ for the symmetric group of permutations on $[p]$. 
	For any $w\in S_p$, we use one-line notation, i.e., $w = (a_1,a_2,\dots,a_p)$ with $w(i)= a_i$. 
	For any $k\in [p]$, we write $w[k]:=\{w(1),w(2),\dots,w(k)\}$. 
	We say $(i,j)$ is an \emph{inversion} of $w$ if $1\leq i <j\leq p$ and $w(i)>w(j)$. 
	For every $i\in [p-1] = \{1,\ldots, p-1\}$, we have the transposition $\mathfrak{s}_i = (i, i+1) \in S_p$.
	The \emph{length} $\ell(w)$ of a permutation $w$ is the number of inversions of $w$ or, equivalently, the minimum number of transpositions $\mathfrak{s}_i$ of which $w$ can be written as a product. 
	
	We equip $S_p$ with a partial order structure known as \emph{Bruhat order}.  For comprehensive references on Bruhat order and related topics, we point to \cite{H90,BB05}.
	
	\begin{definition}
		For $u,w\in S_p$, we say $w\geq u$ in Bruhat order if $w[i] \geq u[i]$ for all $i\in [p]$. 
	\end{definition}

	Double Schubert polynomials were introduced by Lascoux and Sch\"utzenberger \cite{LS_Schubert} and can be defined via divided difference operators.
	\begin{definition}
		The $i$-th divided difference operator $\partial_i$ takes each polynomial $f\in \ZZ[t_1,\dots,t_p]$ to 
		\[
		\partial_i f(t_1,\dots,t_p) \;=\; \frac{f(t_1,\dots,t_p)-f(t_1,\dots,t_{i-1}, t_{i+1},t_i,t_{i+2},\dots,t_p)}{t_i-t_{i+1}}.
		\]
		The \emph{Schubert polynomial} for $w \in S_p$ is defined by the recursion
		\[
		\mathfrak{S}_w(t_1,\dots,t_p) = \partial_i\mathfrak{S}_{w\mathfrak{s}_i}(t_1,\dots,t_p)
		\]
		whenever $\ell(w)<\ell(w\mathfrak{s}_i)$, with initial data $\mathfrak{S}_{w_0}(t_1,\ldots,t_p) = \prod_{i=1}^{p-1}t_i^{p-i}\in \ZZ[t_1,\dots,t_p]$. 
		The \emph{double Schubert polynomials} $\mathfrak{S}_w(\ttt,\sss)$ are defined by the same recursion where the divided difference operators only act on the $t_i$'s, but starting from $\mathfrak{S}_{w_0}(\ttt,\sss)= \prod_{i+j\leq p}(t_i-s_j)\in \ZZ[t_1,\dots,t_p,s_1,\dots,s_p]$.
	\end{definition}

	\subsection{Multidegrees}
	Here we recall the definition and basic properties of multidegrees as presented in \cite{KNUTSON_MILLER_SCHUBERT}  and \cite[Chapter 8]{MS}.
	
	Let $R = \kk[x_1,\ldots,x_n]$ be a $\ZZ^p$-graded polynomial ring over a field $\kk$. 
	Let $M$ be a finitely generated $\ZZ^p$-graded $R$-module and $F_\bullet$ be a $\ZZ^p$-graded free $R$-resolution 
	$
	F_\bullet\colon \; \cdots \rightarrow F_i \rightarrow F_{i-1} \rightarrow \cdots \rightarrow F_1 \rightarrow F_0
	$
	of $M$.
	Let $t_1,\ldots,t_p$ be variables over $\ZZ$ and consider the Laurent polynomial ring $\ZZ[\ttt, \mathbf{t^{-1}}] = \ZZ[t_1,\ldots,t_p,t_1^{-1},\ldots,t_p^{-1}]$, where the variable $t_i$ corresponds with the $i$-th elementary vector $\ee_i \in \ZZ^p$. 
	If we write $F_i = \bigoplus_{j} R(-\mathbf{b}_{i,j})$ with $\mathbf{b}_{i,j} = (\mathbf{b}_{i,j,1},\ldots,\mathbf{b}_{i,j,p}) \in \ZZ^p$, then we define the Laurent polynomial 
	$
	\left[F_i\right]_\ttt \, := \, \sum_{j} \ttt^{\mathbf{b}_{i,j}} = \sum_{j} t_1^{\mathbf{b}_{i,j,1}} \cdots t_p^{\mathbf{b}_{i,j,p}}.
	$
	
	\begin{definition}
		The \emph{$K$-polynomial} of $M$ is the Laurent polynomial given by 
		$$
		\mathcal{K}(M;\ttt) \; := \; \sum_{i} {(-1)}^i \left[ F_i \right]_\ttt.
		$$
	\end{definition}
	
	We have that, even if the grading of $R$ is non-positive and we do not have a well-defined notion of Hilbert series, the above definition of $K$-polynomial is an invariant of the module $M$ and it does not depend on the chosen free $R$-resolution $F_\bullet$ (see \cite[Theorem 8.34]{MS}).
	
	\begin{definition}
		The \emph{multidegree polynomial} of $M$ is the homogeneous polynomial $\mathcal{C}(M; \ttt) \in \ZZ[\ttt]$ given as the sum of all terms in 
		$\mathcal{K}(M; \mathbf{1} - \ttt) = \mathcal{K}(M; 1-t_1,\ldots,1-t_p)$ having total degree $\codim(M)$, which is the lowest degree appearing.
	\end{definition}

	\subsection{Lorentzian and covolume polynomials}\label{sec:lorandcovol}
	In this subsection, we briefly recall Lorentzian polynomials, dually Lorentzian polynomials, and covolume polynomials.
	
	A subset $J \subset \NN^p$ is called \emph{$\mathrm{M}$-convex} if for any $q = (q_1,\dots,q_p)$ and $r = (r_1,\dots,r_p)$ in $J$, and any $i$ where $q_i < r_i$, there exists $j$ such that $q_j> r_j$ and both points $q + \ee_i - \ee_j$ and $r -\ee_i + \ee_j$ are also contained in $J$. 
	We point out that M-convex sets are equivalent to basis elements of  \emph{discrete polymatroids} \cite{Murota_Discrete_Convex_Analysis, Schrijver} and integer points of \emph{integral generalized permutohedra} \cite{Postnikov2009}. 
	A polynomial with M-convex support has \emph{saturated Newton polytope} (SNP). 
	In \cite{MTY2019}, Monical, Tokcan, and Yong conjectured that many polynomials in algebraic combinatorics have saturated Newton polytopes.
	
	Let $h(t_1,\dots,t_p)$ be a homogeneous polynomial of degree $d$ in $\RR[\ttt]=\RR[t_1,\ldots,t_p]$.
	\begin{definition}[\cite{LORENTZIAN}]
		The homogeneous polynomial $h$ is called \emph{Lorentzian} if the following conditions hold:
		\begin{enumerate}[\rm (i)]
			\item The coefficients of $h$ are nonnegative.
			\item The support of $h$ is M-convex.
			\item The quadratic form $\frac{\partial}{\partial t_{i_1}}\frac{\partial}{\partial t_{i_2}}\cdots \frac{\partial}{\partial t_{i_e}}h$ has at most one positive eigenvalue for any $i_1,\ldots,i_e\in [p]$ where $e = d-2$.
		\end{enumerate}		
	\end{definition}
	There are several linear operators discussed in \cite{LORENTZIAN} that preserve the Lorentzian property. In particular, the normalization operator 
	$$
	N\left(\sum_{\bn} a_\bn\ttt^\bn\right) \; := \;\sum_{\bn} \frac{a_\bn}{\bn!}\ttt^\bn \qquad \text{ where }  \qquad \bn! := n_1!\cdots n_p!,
	$$ 
	preserves the Lorentzian property \cite[Corollary 3.7]{LORENTZIAN}.
	A polynomial $h$ is said to be \emph{denormalized Lorentzian} if its normalization $N(h)$ is Lorentzian.
	
	In \cite{ALUFFI_COVOL}, Aluffi defined the notion of covolume polynomials. 
	These polynomials arise by considering the Chow classes of irreducible subvarieties of a product of projective spaces.
	
	\begin{definition}[\cite{ALUFFI_COVOL}]
		\label{def_covol}
		Let $\PP = \PP_\kk^{m_1}\times_\kk\cdots\times_\kk\PP_{\kk}^{m_p}$ be a multiprojective space over a field $\kk$.
		Let $X \subset \PP$ be an irreducible subvariety of codimension $c$. 
		The class of $X$ can be written as 
		$$
		[X] \;=\; \sum_{|\bn| = c}a_\bn \, H_1^{n_1}\cdots H_p^{n_p} \;\in\; A^\bullet(\PP) = \ZZ[H_1,\dots, H_p]/\left( H_1^{m_1+1},\dots, H_p^{m_p+1}\right).
		$$
		Let $P_{[X]}(t_1,\ldots,t_p) = \sum_{|\bn| = c}a_\bn\ttt^{\bn}$ be the polynomial associated to the class $[X]$ of $X \subset \PP$.
		A polynomial $P(t_1,\dots,t_p) \in \RR[t_1,\ldots,t_p]$ with nonnegative real coefficients is a \emph{covolume polynomial}
		if it is a coefficient-wise limit of polynomials of the form $cP_{[X]}$ for a positive real number $c$ and an irreducible subvariety $X$ of $\PP$.
	\end{definition}

	We are interested in the family of dually Lorentzian polynomials introduced by Ross, S\"u\ss, and Wannerer \cite{RSW23}.
	
	\begin{definition}[\cite{RSW23}]
		A polynomial $h\in \RR[t_1,\dots, t_p]$ is \emph{dually Lorentzian} if 
		$$
		N\left(t_1^{m_1}\cdots t_p^{m_p} h\left(t_1^{-1},\dots,t_p^{-1}\right)\right)
		$$ 
		is Lorentzian for some $\bm = (m_1,\ldots,m_p)$ sufficiently large (i.e., such that $t_1^{m_1}\cdots t_p^{m_p} h\left(t_1^{-1},\dots,t_p^{-1}\right)$ is a polynomial).
	\end{definition}
	
	As shown by Aluffi \cite[Proposition 2.8]{ALUFFI_COVOL}, covolume polynomials form a subfamily of the family of dually Lorentzian polynomials.
	
	\begin{remark}
		\label{rem_dual_Lorentzian_monomial_product}
		For any monomial $t_1^{\alpha_1}\cdots t_p^{\alpha_p}$ and any polynomial $h \in \RR[t_1,\ldots,t_p]$, we have that $h$ is dually Lorentzian if and only if $t_1^{\alpha_1}\cdots t_p^{\alpha_p}h$ is dually Lorentzian.  
	\end{remark}
	\begin{proof}
		This follows directly from \cite[Lemma 7]{HMMSD}.
	\end{proof}
	
	Finally, we introduce the following definition.
	
	\begin{definition}
		For any polynomial $h = \sum_\bn a_\bn\ttt^\bn \in \RR[t_1,\ldots,t_p]$ and any $\bw \in \NN^p$, the \emph{$\bw$-truncation} of $h$ is the polynomial given by 
		$
		\sum_{\bn \le \bw} a_\bn\ttt^\bn \in \RR[t_1,\ldots,t_p].
		$
	\end{definition}
	
	\section{Multidegree polynomials of prime ideals in non-standard gradings}\label{sec_pos_gradings}
	
	In this section, working over an arbitrary positive $\NN^p$-grading, we show that the multidegree polynomial of a prime ideal is a covolume polynomial.
	Our main tool is the technique of \emph{standardization} that was used in \cite{DOUBLE_SCHUBERT_POLYM, MULTDEG_NON_STD}.
	
	Let $\kk$ be an algebraically closed field and $R = \kk[x_1,\ldots,x_n]$ be a \emph{positively $\NN^p$-graded} polynomial ring (that is, $\deg(x_i) \in \NN^p \setminus \{\mathbf{0}\}$ for all $1 \le i \le n$ and $\deg(\alpha) = \mathbf{0} \in \NN^p$ for all $\alpha \in \kk$).  
	Let $M$ be a finitely generated $\ZZ^p$-graded $R$-module.
	Since $R$ is assumed to be positively graded, there is a well-defined notion of Hilbert series 
	$$
	\Hilb_M(\ttt) \;=\; \Hilb_M(t_1,\ldots,t_p)  \;:= \; \sum_{\bn \in \ZZ^p} \dim_\kk\left(\left[M\right]_\bn\right) \ttt^\bn = \sum_{\bn = (n_1,\ldots,n_p) \in \ZZ^p}\dim_\kk\left(\left[M\right]_\bn\right) t_1^{n_1}\cdots t_p^{n_p},
	$$
	and then we can write
	$$
	\Hilb_M(\ttt) \;= \;  \frac{\mathcal{K}(M;\ttt)}{\prod_{i=1}^n \left(1 - \ttt^{\deg(x_i)}\right)}.
	$$

	Below we discuss the case of a standard multigrading.
	This case is of special importance since it deals with closed subschemes of a product of projective spaces.

	\begin{remark}[Standard multigradings]
		\label{rem_std_grading}
		Assume that $R$ has a standard $\NN^p$-grading (i.e., $|\deg(x_i)|=1$) and that $R$ is the coordinate ring of $\PP = \PP_\kk^{m_1} \times_\kk \cdots \times_\kk \PP_\kk^{m_p}$.
		Let $X \subset \PP$ be a $d$-dimensional integral closed subscheme with coordinate ring $R/P$, where $P$ is an $R$-homogeneous prime ideal.
		The class of $X$ in the Chow ring of $\PP$ is given by
		$$
		\left[X\right] \;=\; \sum_{\substack{\bn=(n_1,\ldots,n_p)\in \NN^p\\ |\bn|=n_1+\cdots+n_p=d}} \deg_\PP^\bn(X)\cdot H_1^{m_1-n_1}\cdots H_p^{m_p-n_p}  \;\in\;  A^\bullet(\PP) = \frac{\ZZ[H_1,\ldots,H_p]}{\big(H_1^{m_1+1},\ldots, H_p^{m_p+1}\big)},
		$$
		where $H_i$ represents the class of the inverse image of a hyperplane of $\PP_\kk^{m_i}$.
		We say that $\deg_\PP^\bn(X)$ is the multidegree of $X$ of type $\bn$.
		Then the multidegree polynomial of $R/P$ is given by 
		$$
		\mathcal{C}(R/P; \ttt) \;=\; \sum_{\bn\in \NN^p,\; |\bn|=d} \deg_\PP^\bn(X)\cdot t_1^{m_1-n_1}\cdots t_p^{m_p-n_p}.
		$$
		The \emph{volume polynomial} of $X$ (see \cite[\S 4.2]{LORENTZIAN}) is given by 
		$$
		{\rm vol}_X(\ttt) \;=\; \int\big(H_1t_1 + \cdots + H_pt_p\big)^{d} \cap \left[X\right] \;=\; \sum_{\bn\in \NN^p,\; |\bn|=d} \deg_\PP^\bn(X)\cdot\frac{d!}{n_1!\cdots n_p!} \cdot t_1^{n_1}\cdots t_p^{n_p}.  
		$$
		Due to a fundamental result of Br\"and\'en and Huh (see \cite[Theorem 4.6]{LORENTZIAN}), we have that $\vol_X(\ttt)$ is a Lorentzian polynomial.
	\end{remark}
	
	The following lemma tells us that multidegree polynomials in a standard multigrading are \emph{dually Lorentzian} \cite{RSW23}.
	This result already appeared in \cite[Proposition 2.8]{ALUFFI_COVOL} (also, see \cite[Theorem 6]{HMMSD}), but here we give a short self-contained proof.
	
	\begin{lemma}
		\label{lem_std_grading}
		Keep the same notations and assumptions of \autoref{rem_std_grading}.
		Consider the polynomial
		$$
		F(t_1,\ldots,t_p) \;=\; t_1^{m_1}\cdots t_p^{m_p} \cdot \mathcal{C}\Big(R/P; \frac{1}{t_1}, \ldots, \frac{1}{t_p}\Big).
		$$
		Then the normalization $N(F)$ is a Lorentzian polynomial {\rm(}i.e., $\mathcal{C}(R/P;\ttt)$ is dually Lorentzian{\rm)}.
	\end{lemma}
	\begin{proof}
		From \autoref{rem_std_grading}, we derive the following equalities 
		\begin{align*}
			N(F(\ttt)) &\;=\; N\Bigg(t_1^{m_1}\cdots t_p^{m_p} \cdot \mathcal{C}\Big(R/P; \frac{1}{t_1}, \ldots, \frac{1}{t_p}\Big)\Bigg) \\
			&\;=\; N\Bigg(t_1^{m_1}\cdots t_p^{m_p} \cdot \sum_{\bn\in \NN^p,\; |\bn|=d} \deg_\PP^\bn(X)\cdot \frac{1}{t_1^{m_1-n_1}}\cdots \frac{1}{t_p^{m_p-n_p}}\Bigg) \\
			&\;=\; N\Bigg( \sum_{\bn\in \NN^p,\; |\bn|=d} \deg_\PP^\bn(X)\cdot {t_1^{n_1}}\cdots {t_p^{n_p}}\Bigg) \\
			&\;=\; \sum_{\bn\in \NN^p,\; |\bn|=d} \deg_\PP^\bn(X)\cdot\frac{1}{n_1!\cdots n_p!} \cdot t_1^{n_1}\cdots t_p^{n_p}\\
			&\;=\; \frac{1}{d!} \cdot \vol_X(\ttt).
		\end{align*}
		Therefore \cite[Theorem 4.6]{LORENTZIAN} implies that $N(F)$ is Lorentzian.
	\end{proof}
	
	We now describe some basic properties of the process of standardization as developed in \cite{DOUBLE_SCHUBERT_POLYM, MULTDEG_NON_STD}.
	For the rest of the section, the following setup is fixed.

	\begin{setup}
		\label{setup_std}
		For $1 \le i \le n$, let $\ell_i = |\deg(x_i)|$ be the total degree of the variable $x_i$.
		Let
		$$
		S \;= \; \kk\left[y_{i,j} \mid 1 \le i \le n \text{ and } 1 \le j \le \ell_i\right]
		$$
		be a standard $\NN^p$-graded polynomial ring such that
		$$
		\deg(x_i) \;=\; \sum_{j = 1}^{\ell_i} \deg(y_{i,j})   \quad  \text{for all \;\;$1 \le i \le n$.}
		$$
		We define the $\NN^p$-graded $\kk$-algebra homomorphism 
		\begin{equation*}
			\phi\colon R=\kk[\xx] \longrightarrow S = \kk[\mathbf{y}], \quad
			\phi(x_{i}) = y_{i,1}y_{i,2}\cdots y_{i,\ell_i}.
		\end{equation*}
		For an $R$-homogeneous ideal $I \subset R$, we say that the extension  $J = \phi(I) S$ is the \emph{standardization} of $I$, since $J \subset S$ is an $S$-homogeneous ideal in the standard $\NN^p$-graded polynomial ring $S$.
		By a slight abuse of notation, we consider both multidegree polynomials $\mathcal{C}(R/I;\ttt)$ and $\mathcal{C}(S/J;\ttt)$ as elements of the same polynomial ring $\ZZ[\ttt]=\ZZ[t_1,\ldots,t_p]$.
	\end{setup}

	The next theorem contains some of the basic and desirable properties of the standardization of an ideal.
	
	\begin{theorem}[{\cite[Theorem 7.2]{MULTDEG_NON_STD}, \cite[Proposition 4.2]{DOUBLE_SCHUBERT_POLYM}}]
		\label{thm_std}
		Let $I \subset R$ be an $R$-homogeneous ideal and $J = \phi(I)S$ be its standardization. 
		Then the following statements hold:
		\begin{enumerate}[\rm (i)]
			\item $\codim(I) = \codim(J)$.
			\item $I \subset R$ and $J \subset S$ have the same $\NN^p$-graded Betti numbers. 
			\item $\mathcal{K}(R/I;\ttt) = \mathcal{K}(S/J;\ttt)$ and  $\mathcal{C}(R/I;\ttt) = \mathcal{C}(S/J;\ttt)$.
			\item $R/I$ is a Cohen--Macaulay ring if and only if $S/J$ is a Cohen--Macaulay ring.
			\item Let $>$ be a monomial order on $R$ and $>'$ be a monomial order on $S$ which is compatible with $\phi$ {\rm(}i.e.,~if $f,g \in R$ with $f > g$, then $\phi(f) >' \phi(g)${\rm)}.
			Then $\init_{>'}(J) = \phi(\init_{>}(I))S$.
			\item If $I \subset R$ is a prime ideal and it does not contain any variable, then $J \subset S$ is also a prime ideal.
		\end{enumerate}
	\end{theorem}

	Finally, we are ready to present the main result of this section. 
	It yields a large new family of covolume polynomials.

	\begin{theorem}\label{thm:dually_lor}
		Let $P \subset R$ be a prime $R$-homogeneous ideal. 
		Then $\mathcal{C}(R/P; \ttt)$ is a covolume polynomial.
	\end{theorem}
	\begin{proof}
		Let $\mathcal{L} = \{ i \mid x_i \in P \}$ be the set that indexes the variables belonging to $P$.
		Write $P = P' + \left(x_i \mid i \in \mathcal{L}\right)$ with $P' \subset R$ only involving the variables not in $\mathcal{L}$.
		Let $T = R\left[z_i \mid i \in \mathcal{L}\right]$ be a positively $\NN^p$-graded polynomial ring extending the grading of $R$ and with $\deg(z_i) = \deg(x_i)$.
		Let $\mathfrak{P} = P'T + \left(x_i - z_i \mid i \in \mathcal{L}\right) \subset T$.
		Notice that $\mathfrak{P}$ is a prime ideal containing no variable and that 
		$$
		\mathcal{C}(T/\mathfrak{P}; \ttt) \;=\;  \prod_{i \in \mathcal{L}} \langle \deg(x_i-z_i), \ttt \rangle \cdot \mathcal{C}(T/P'T; \ttt) \;=\; \prod_{i \in \mathcal{L}} \langle \deg(x_i), \ttt \rangle \cdot \mathcal{C}(R/P'; \ttt) \;=\; \mathcal{C}(R/P;\ttt),
		$$
		where $\langle \deg(x_i), \ttt \rangle = a_{i,1}t_1 + \cdots + a_{i,p}t_p \in \NN[\ttt]$ after writing $(a_{i,1}, \ldots,a_{i,p}) = \deg(x_i) \in \NN^p$ (see \cite[Exercise 8.12]{MS}).
		Let $W = S\left[w_{i,j} \mid i \in \mathcal{L}, 1 \le j \le \ell_i \right]$ be a standard $\NN^p$-graded polynomial ring where we consider the corresponding standardization $\mathfrak{Q} \subset W$ of $\mathfrak{P} \subset T$.
		From \autoref{thm_std}(iii), we get 
		$$
		\mathcal{C}(R/P;\ttt) \;=\; \mathcal{C}(T/\mathfrak{P};\ttt) \;=\; \mathcal{C}(W/\mathfrak{Q}; \ttt).
		$$
		Finally, $\mathfrak{Q} \subset W$ is a prime ideal by \autoref{thm_std}(vi), and so the result follows by \autoref{rem_std_grading}.
	\end{proof}

	\section{Equivariant cohomology in multigraded commutative algebra}\label{s:Equiv_coho_multigraded}
	
	In this section, we study the equivariant cohomology of the irreducible varieties that appear in multigraded commutative algebra. 
	Here we show that, under suitable conditions, equivariant classes of irreducible multigraded varieties yield covolume polynomials (up to changing the sign of negative coefficients). 
	We follow the references \cite{AF_Equi_Coho_AG} and \cite[Chapter 5]{REP_THEORY_GEOM} for the basics of equivariant cohomology and equivariant $K$-theory.
	We note that to work over fields other than $\CC$, one may use equivariant intersection theory \cite{EG_equiv}.
	
	Let $R=\CC[x_1,\ldots,x_n]$ be a polynomial ring with  $\ZZ^p$-grading given by 
	$$\deg(x_i) \;=\; \dd_i \;=\; (d_{i,1}, d_{i,2}, \ldots, d_{i,p}) \;\in\; \ZZ^p
	$$ 
	for all $1 \le i \le n$.
	Consider the corresponding torus $T = (\CC^*)^p$ acting on $\AAA_\CC^n$ with action
	$$
	\mathbf{h} \cdot \mathbf{a} \;=\; \big(h_1^{d_{1,1}}h_2^{d_{1,2}}\cdots h_p^{d_{1,p}} a_1,\, h_1^{d_{2,1}}h_2^{d_{2,2}}\cdots h_p^{d_{2,p}} a_2,\, \ldots,\, h_1^{d_{n,1}}h_2^{d_{n,2}}\cdots h_p^{d_{n,p}} a_n\big) 
	$$
	for all $\mathbf{h} = (h_1,\ldots,h_p) \in T$ and $\mathbf{a} = (a_1,\ldots,a_n) \in \AAA_\CC^n$.
	There is a bijective correspondence between $\ZZ^p$-graded $R$-modules and $T$-equivariant coherent sheaves on $\AAA_\CC^n$.
	Since $\AAA_\CC^n$ is smooth, the Grothendieck ring of $T$-equivariant vector bundles coincides with the Grothendieck ring of $T$-equivariant coherent sheaves. 
	We denote this ring by $K_T(\AAA_\CC^n)$.
	The Grothendieck ring $K_T(\AAA_\CC^n)$ coincides with the Grothendieck ring of a point and the representation ring of $T$:
	$$
	K_T(\AAA_\CC^n) \;\cong\; K_T({\rm pt}) \;\cong\; R(T) \;=\; \ZZ[t_1,\ldots, t_p,t_1^{-1},\ldots, t_p^{-1}].
	$$
	Given a finitely generated $\ZZ^p$-graded $R$-module $M$ (i.e., $\widetilde{M}$ is a $T$-equivariant coherent sheaf on $\AAA_\CC^n$), we denote by $[M]^T$ the equivariant class of $M$ in $K_T(\AAA_\CC^n) \cong \ZZ[t_1,\ldots, t_p,t_1^{-1},\ldots, t_p^{-1}]$.

	We consider the $T$-equivariant cohomology ring 
	$$
	\HH_T^\bullet(\AAA_\CC^n) \;:=\;  \HH^\bullet\big(\mathbb{E}T \times^T \AAA_\CC^n\big),
	$$
	where $\mathbb{E}T$ is contractible with $T$ acting freely (on the right). 
	Then $\mathbb{B}T := \mathbb{E}T/T$ is a classifying space for $T$.
	Since we can take $\mathbb{E}T=\left(\CC^\infty \setminus \{0\}\right)^p$ and $\mathbb{B}T = \left(\PP_\CC^\infty\right)^p$, it follows that 
	$$
	\HH_T^\bullet(\AAA_\CC^n) \;\cong\; \Lambda_T := \HH^\bullet_T({\rm pt}) \;=\; \HH^\bullet(\mathbb{B}T) \;=\; \ZZ[t_1,\ldots,t_p].
	$$
	Given a $T$-subvariety $X \subset \AAA_\CC^n$, we denote by $[X]^T := [\mathbb{E}T \times^T X]$ the equivariant class of $X$ in $\HH_T^\bullet(\AAA_\CC^n) \cong \ZZ[t_1,\ldots,t_p]$.
	
	The following remark allows us to apply our results in \autoref{sec_pos_gradings} to equivariant cohomology. 
	
	\begin{remark}
		\label{rem_equiv_multdeg}
		Let $M$ be a finitely generated $\ZZ^p$-graded $R$-module and $X \subset \AAA_\CC^n$ be a $T$-subvariety with coordinate ring $R/I$.  
		Let $K_M(t_1,\ldots,t_p)$ and $C_X(t_1,\ldots,t_p)$ be the Laurent polynomial and the polynomial representing the classes $[M]^T \in K_T(\AAA_\CC^n)$ and $[X]^T \in \HH_T^\bullet(\AAA_\CC^n)$, respectively.
		Then we have the equalities
		$$
		K_M(t_1,\ldots,t_p) \;=\; \mathcal{K}(M;t_1,\ldots,t_p) \qquad \text{ and } \qquad C_X(t_1,\ldots,t_p) \;=\; \mathcal{C}(R/I;t_1,\ldots,t_p).
		$$
	\end{remark}
	\begin{proof}
		See, e.g., \cite{KNUTSON_MILLER_SCHUBERT} or \cite[Proposition 1.19]{KMS_QUIVER}.
	\end{proof}
	
	Then \autoref{thm:dually_lor} poses the question of whether the class $[X]^T$ of any irreducible $T$-subvariety $X \subset \AAA_\CC^n$  yields a covolume polynomial.
	The following example shows that it is not always the case for arbitrary $T$-actions. 
	
	\begin{example}
		\label{examp_bad_grading}
		Consider $\AAA_\CC^2$ and the torus $T = (\CC^*)^2$ with the action 
		$$
		(h_1, h_2) \cdot (a_1, a_2) \;:=\; (h_1h_2 a_1, h_1h_2^{-1}a_2)
		$$
		for all $(a_1,a_2) \in \AAA_\CC^2$ and $(h_1,h_2) \in T$.
		Let $R = \CC[x,y]$ be the coordinate ring of $\AAA_\CC^2$. 
		The $T$-action above specifies the grading $\deg(x) = (1, 1) \in \ZZ^2$ and $\deg(y) = (1,-1) \in \ZZ^2$.
		Consider the variety $X = \{(0,0)\} \subset \AAA_\CC^2$ given by the origin. 
		It is clearly irreducible, but the polynomial representative of the cohomology class
		$$
		[X]^T \;=\; \cC(R/(x,y); t_1, t_2) \;=\; (t_1+t_2)(t_1-t_2) \;=\; t_1^{2} - t_2^{2}  \qquad \text{(see, e.g., \cite[Proposition 8.49]{MS})}
		$$
		is not a covolume polynomial (its support is not $M$-convex) after changing signs.
	\end{example}

	For the rest of this section, we shall use the following setup that avoids the complications of the previous example.
	
	\begin{setup}
		\label{setup_almost_pos}
		Assume that the torus $T$ is given as $T = (\CC^*)^{q} \times (\CC^*)^{p-q}$ where ``$ (\CC^*)^{q}$ comes with positive weights and $(\CC^*)^{p-q}$ comes with negative weights''.
		More precisely, we require that $\dd_i \in \ZZ^p \setminus \{\mathbf{0}\}$, 
		$$
		d_{i,j} \in \NN  \quad \text{ for all $1 \le j \le q$}
		$$
		and 
		$$
		-d_{i,j} \in \NN  \quad \text{ for all $q+1 \le j \le p$}.
		$$
		In this case we say that the action of $T$ determines a \emph{twisted positive grading} on $R$.
	\end{setup}

	The following lemma tells us that we can ``flip'' twisted positive gradings to positive gradings. 
	A version of this lemma appeared in \cite[Lemma 3.3]{DOUBLE_SCHUBERT_POLYM}.
	Consider the polynomial ring $\widetilde{R} = \CC[x_1,\ldots,x_n]$ over the field $\CC$ with the same set of variables as $R$ but with grading given by setting 
	$$
	\deg(x_i) = (d_{i,1},\ldots,d_{i,q}, -d_{i,q+1},\ldots,-d_{i,p}).$$
	Notice that $\widetilde{R}$ has a positive $\NN^p$-grading.
	
	\begin{lemma}
		\label{lem_flip_grading}
		Assume \autoref{setup_almost_pos}. 
		Let $I \subset R$ be a $\ZZ^p$-graded ideal, and also denote by $I$ the corresponding $\NN^p$-graded ideal in $\widetilde{R}$.
		Then we have $\mathcal{C}(\widetilde{R}/I; t_1,\ldots,t_q, t_{q+1},\ldots,t_p)  = \mathcal{C}(R/I; t_1,\ldots,t_q, -t_{q+1},\ldots,-t_p).$
	\end{lemma}
	\begin{proof}
		Let $r = p-q$ and set $s_i = t_{q+i}$ for $1 \le i \le r$.
		Notice that, if $\widetilde{F}_\bullet$ is a $(\ZZ^q \oplus \ZZ^{r})$-graded free $\widetilde{R}$-resolution of $\widetilde{R}/I$ with $\widetilde{F}_i = \bigoplus_{j} \widetilde{R}(-\mathbf{a}_{i,j},-\mathbf{b}_{i,j})$, then there is a corresponding  $(\ZZ^q \oplus \ZZ^{r})$-graded free $R$-resolution $F_\bullet$ of $R/I$ with $F_i = \bigoplus_{j} R(-\mathbf{a}_{i,j},\mathbf{b}_{i,j})$.
		By definition, this yields the equality of $K$-polynomials
		$$
		\mathcal{K}(\widetilde{R}/I;\ttt,\sss) = \mathcal{K}(\widetilde{R}/I;t_1,\ldots,t_q,s_1,\ldots,s_r) = \mathcal{K}(R/I;t_1,\ldots,t_q,s_1^{-1},\ldots,s_r^{-1}) = 	\mathcal{K}(R/I;\ttt,\sss^{\mathbf{-1}}).
		$$
		From \cite[Claim 8.54]{MS}, we have $\mathcal{K}(R/I;\mathbf{1-t},\mathbf{1-s}) = \mathcal{C}(R/I;\ttt,\sss) + Q(\ttt,\sss)$, where $Q(\ttt,\sss)$ is the sum of the terms of degree at least $\codim(I) + 1$.
		Equivalently, we get $\mathcal{K}(R/I;\ttt,\sss) = \mathcal{C}(R/I;\mathbf{1-t},\mathbf{1-s}) + Q(\mathbf{1-t},\mathbf{1-s})$. 
		It then follows that 
		$$
		\mathcal{K}(\widetilde{R}/I;\mathbf{1-t},\mathbf{1-s}) = \mathcal{C}(R/I; t_1,\ldots,t_q,1-\tfrac{1}{1-s_1},\ldots,1-\tfrac{1}{1-s_r}) + Q(t_1,\ldots,t_q,1-\tfrac{1}{1-s_1},\ldots,1-\tfrac{1}{1-s_r}).
		$$
		By expanding the right hand side of the above equality, the result of the lemma is obtained.	
	\end{proof}

	We are now ready to present the main result of this section.

	\begin{theorem}
		\label{thm_equiv}
		Assume \autoref{setup_almost_pos}. 
		Let $X \subset \AAA_\CC^n$ be an irreducible $T$-subvariety and $C_X(t_1,\ldots,t_p)$ be the polynomial representing the class $[X]^T \in \HH_T^\bullet(\AAA_\CC^n)$.
		Then $C_X(t_1,\ldots,t_q, -t_{q+1},\ldots,-t_p)$ is a covolume polynomial.
	\end{theorem}
	\begin{proof}
		The result follows by combining \autoref{rem_equiv_multdeg}, \autoref{lem_flip_grading}, and \autoref{thm:dually_lor}.
	\end{proof}
	
	\begin{corollary}[{cf.~\cite[Corollary 2.14]{ALUFFI_COVOL}}]
		\label{cor_product}
		Assume \autoref{setup_almost_pos}. 
		Let $X \subset \AAA_\CC^n$ and $Y \subset \AAA_\CC^n$ be  irreducible $T$-subvarieties,  and $C_X(t_1,\ldots,t_p)$ and $C_Y(t_1,\ldots,t_p)$ be the polynomials representing the classes $[X]^T \in \HH_T^\bullet(\AAA_\CC^n)$ and $[Y]^T \in \HH_T^\bullet(\AAA_\CC^n)$, respectively.
		Then 
		$$
		C_X(t_1,\ldots,t_q, -t_{q+1},\ldots,-t_p) \cdot C_Y(t_1,\ldots,t_q, -t_{q+1},\ldots,-t_p)
		$$ 
		is a covolume polynomial.
	\end{corollary}
	\begin{proof}
		Consider the product $\AAA_\CC^n \times \AAA_\CC^n = \Spec\left(\CC[y_1,\ldots,y_n] \otimes_\CC \CC[z_1,\ldots,z_n]\right)$ with the diagonal $T$-action (i.e., for any $\mathbf{h} \in T$ and $(\mathbf{a}, \mathbf{b}) \in \AAA_\CC^n \times \AAA_\CC^n$, we have $\mathbf{h} \cdot (\mathbf{a}, \mathbf{b}) = (\mathbf{h} \cdot \mathbf{a}, \mathbf{h} \cdot \mathbf{b})$). 
		Thus the polynomial ring $\CC[y_1,\ldots,y_n] \otimes_\CC \CC[z_1,\ldots,z_n]$ is $\ZZ^p$-graded with $\deg(y_i) = \deg(z_i) = \deg(x_i) = \dd_i$.
		Consequently, under the isomorphisms $\HH_T^\bullet\left(\AAA_\CC^n\right) \cong \ZZ[t_1,\ldots,t_p] \cong \HH_T^\bullet\left(\AAA_\CC^n \times \AAA_\CC^n\right)$, we may identify $[X]^T$ and $[Y]^T$ with $[X \times \AAA_\CC^n]^T$ and $[\AAA_\CC^n \times Y]^T$, respectively.
		By construction  $X \times \AAA_\CC^n$ and $\AAA_\CC^n \times Y$ intersect transversely, and so we obtain
		$$
		\left[X \times \AAA_\CC^n\right]^T \cdot \left[\AAA_\CC^n \times Y\right]^T \;=\; \left[\left(X \times \AAA_\CC^n\right) \cap \left(\AAA_\CC^n \times Y\right)\right]^T \;=\; \left[X \times Y\right]^T \;\in\; \HH_T^\bullet\left(\AAA_\CC^n \times \AAA_\CC^n\right).
		$$ 
		The result now follows by applying \autoref{thm_equiv} to the irreducible subvariety $X \times Y \subset \AAA_\CC^n \times \AAA_\CC^n$.
	\end{proof}
	
	Our main application of the above theorem is the following corollary. 
	
	\begin{corollary}
		\label{c:cor_equiv_matrices}
		Let ${\rm Mat}_{m,n} = \CC^{m \times n}$ be the space of $m \times n$ matrices with complex entries and consider the natural action of the torus $T = (\CC^*)^m \times (\CC^*)^n$ given by $(g,h) \cdot M = g \cdot M \cdot h^{-1}$ for all $M \in {\rm Mat}_{m,n}$ and $(g,h) \in T$.
		Let $X \subset {\rm Mat}_{m,n}$ be an irreducible $T$-variety and $C_X(t_1,\ldots,t_m,s_1,\ldots,s_n)$ be the polynomial representing the class $[X]^T$ in $\HH_T^\bullet\left({\rm Mat}_{m,n}\right) = \ZZ[t_1,\ldots,t_m,s_1,\ldots,s_n]$.
		Then we have that $C_X(t_1,\ldots,t_m,-s_1,\ldots,-s_n)$ is a covolume polynomial.
	\end{corollary}

	\section{Equivariant cohomology and double Schubert polynomials}
	\label{sect_equiv_cohom_Richardson}
	
	In this section, we define a new family of polynomials that we call \emph{double Richardson polynomials}.
	These polynomials specialize to many polynomials of interest. 
	We take the viewpoint of studying Schubert varieties and Schubert polynomials through a certain degeneracy loci problem that was introduced by Fulton \cite{Fulton1992}.
	
	Let $E_\bullet$ and $F_\bullet$ be two flagged complex vector spaces of dimension $n$ with 
	$$
	F_1 \subset \cdots \subset F_n = F \qquad \text{ and } \qquad E = E_n \surjects \cdots \surjects E_1,
	$$ 
	where $\dim(F_i) = \dim(E_i) = i$.
	By choosing bases, we identify $\Hom(F, E)$ with the space $\Mat_{n,n} \cong \CC^{n \times n}$ of $n \times n$ matrices with complex entries.
	Consider the natural action of the torus $T = (\CC^*)^n \times (\CC^*)^n$ given by $(g,h) \cdot M = g \cdot M \cdot h^{-1}$ for all $M \in {\rm Mat}_{n,n}$ and $(g,h) \in T$.
	For a given permutation $w \in S_n$, we consider the degeneracy locus 
	$$
	D_w \;:=\; \big\lbrace \varphi \in \Hom(F, E) \mid \rank(\varphi_{p,q}) \le \rank(w_{p,q}) \text{ for all $1 \le p,q \le n$}\big\rbrace,
	$$
	where $\varphi_{p,q}$ is the composition of maps $F_q \rightarrow F \xrightarrow{\varphi} E \rightarrow E_p$ and $w_{p\times q}$ is the upper-left submatrix of size $p\times q$ of the permutation matrix of $w$.
	Under the identification $\Hom(F, E) \cong \Mat_{n,n}$, we say that $D_w$ is the \emph{matrix Schubert variety} of $w$. 
	
	Choose a basis $\{f_1,\ldots,f_n\}$ of $F$ such that $\{f_1,\ldots,f_i\}$ is a basis of $F_i$ for all $1 \le i \le n$.
	Let $\widetilde{F}_\bullet$ be the opposite flag 
	$$
	\widetilde{F}_1 \subset \cdots \subset \widetilde{F}_n \qquad \text{ where } \qquad \widetilde{F}_i = \langle f_{n-i+1}, \ldots, f_n \rangle.
	$$
	For a given permutation $w \in S_n$, we consider the degeneracy locus 
	$$
	D^w \;:=\; \big\lbrace \varphi \in \Hom(F, E) \mid \rank(\varphi^{p,q}) \le \rank(w^{p,q}) \text{ for all $1 \le p,q \le n$}\big\rbrace,
	$$
	where $\varphi^{p,q}$ is the composition of maps $\widetilde{F}_q \rightarrow F \xrightarrow{\varphi} E \rightarrow E_p$ and $w^{p\times q}$ is the upper-right submatrix of size $p\times q$ of the permutation matrix of $w$.
	Under the identification $\Hom(F, E) \cong \Mat_{n,n}$, we say that $D^w$ is the \emph{opposite matrix Schubert variety} of $w$.

	A result of fundamental importance for us is the following geometric interpretation of double Schubert polynomials (see \cite{AF_Equi_Coho_AG} for a comprehensive treatment). 
	
	\begin{theorem}[{\cite{FR,KNUTSON_MILLER_SCHUBERT}}]
		\label{thm_double_Schubert_equiv}
		The equivariant class $\left[D_w\right]^T$ of the matrix Schubert variety $D_w$ in the cohomology ring $\HH_T^\bullet(\Mat_{n,n}) \cong \ZZ[t_1,\ldots,t_n,s_1,\ldots,s_n]$ is given by the double Schubert polynomial $\fS_w(\ttt,\sss)$.
	\end{theorem}
	
	The next lemma expresses the equivariant class of the degeneracy locus $D^w$.
	
	\begin{lemma}\label{l:cohomology_class_opp_sch}
		The equivariant class $\left[D^w\right]^T \in \HH_T^\bullet(\Mat_{n,n})$ of $D^w$ is given by the double Schubert polynomial 
		$$
		\fS_{w_0w}(t_1,\ldots,t_n,s_n,\ldots,s_1)
		$$ 
		of $w_0w$ with a reverse ordering of the variables $s_1,\ldots,s_n$. 
	\end{lemma}
	\begin{proof}
		Let $\alpha\colon F \rightarrow F$ be the involution given by setting $\alpha(f_i) = f_{n+1-i}$.
		To simplify notation, let $\mathcal{M} = \Hom(F, E)$.
		Let $\eta\colon T \rightarrow T$ be the involution given by $\eta(g_1,\ldots,g_n, h_1,\ldots,h_n) = (g_1,\ldots,g_n, h_n,\ldots,h_1)$.
		Then the induced involution $f\colon \mathcal{M} \rightarrow \mathcal{M}, \varphi \mapsto \varphi \circ \alpha$ is an equivariant map with respect to $\eta$ because 
		$$
		f\left((g,h) \cdot \varphi\right) \;=\; \eta(g,h) \cdot f(\varphi)
		$$
		for all $\varphi \in \mathcal{M}$ and $(g, h) \in T$.
		Then the induced pullback homomorphism is given by 
		$$
		f^*\colon\; \HH_T^\bullet(\mathcal{M}) = \ZZ[\ttt, \sss] \;\rightarrow\; \HH_T^\bullet(\mathcal{M}) = \ZZ[\ttt, \sss], \qquad t_i \mapsto t_i \text{\; and \;} s_j \mapsto s_{n+1-j}
		$$
		(see \cite[Exercise 3.2.3]{AF_Equi_Coho_AG}).
		Let $\mathbb{E}_m = (\CC^m\setminus 0)^n \times (\CC^m\setminus 0)^n$.
		Then $\HH_T^i(\mathcal{M}) = \HH^i(\mathbb{E}_m \times^T \mathcal{M})$ for $i < 2m-1$
		and the pullback map $f^*$ can be constructed by taking the pullback in (usual) cohomology along the equivariant map
		$$
		\mathbb{E}_m \times^T \mathcal{M}  \;\xrightarrow{(\gamma, f)} \; \mathbb{E}_m \times^T \mathcal{M}
		$$
		where $\gamma\colon \mathbb{E}_m \rightarrow \mathbb{E}_m$ is the involution given by $\gamma(u_1,\ldots,u_n,v_1,\ldots,v_n)=(u_1,\ldots,u_n,v_n,\ldots,v_1)$ for all $(u_1,\ldots,u_n,v_1,\ldots,v_n) \in \mathbb{E}_m$ (see \cite[Exercise 3.2.1]{AF_Equi_Coho_AG}).
		Since $\mathbb{E}_m \times^T \mathcal{M}$ is smooth, it follows that 
		$$
		f^*\left(\big[D_{w_0w}\big]^T\right) \;=\; \big[f^{-1}(D_{w_0w})\big]^T = \big[D^{w}\big]^T
		$$
		(see, e.g., \cite[Proposition A.3.2]{AF_Equi_Coho_AG}, \cite[\S B.3]{Fulton_Young_Tableaux}).
		The result of the lemma now follows by invoking \autoref{thm_double_Schubert_equiv}.
	\end{proof}

	We are now ready to introduce the  main object of this section.
	
	\begin{definition}\label{d:matrix_Richardson_Var}
		For a pair of permutations $(w, u)$ in $S_n$, we have the \emph{double Richardson polynomial} $\fR_{w/u}(\ttt,\sss):=\fS_u(\ttt,\sss)\fS_{w_0 w}(\ttt,\sss')$, where $\sss' = (s_n,\ldots,s_1)$ denotes the reverse of $\sss = (s_1,\ldots,s_n)$. 
	\end{definition}

	\begin{remark}
		\label{rem_Richardson}
		Let $\Fl$ be the full flag variety with its natural torus action $T'= (\CC^*)^n \subset \GL_n(\CC)$.
		The equivariant cohomology of $\Fl$ is given by 
		$$
		\HH_{T'}^\bullet\left(\Fl\right) \;\cong\; \frac{\ZZ[t_1,\ldots,t_n,s_1,\ldots,s_n]}{\left(e_1(\ttt)-e_1(\sss), \ldots, e_n(\ttt)-e_n(\sss)\right)}
		$$
		(where $e_i(\ttt)$ and $e_i(\sss)$ denote the $i$-th elementary symmetric polynomials in $\ttt$ and $\sss$, respectively).		
		Let $w \ge u$ in the  Bruhat order.							
		The Richardson variety $R_u^w \subset \Fl$ is given by the transversal intersection $R_u^w := X_u \cap X^w$ (see \cite[\S 1.5]{Speyer2024}).
		We have that $R_u^w$ is an irreducible Cohen-Macaulay  projective variety with rational singularities and of dimension $\ell(w) - \ell(u)$.
		Moreover, inside the equivariant cohomology ring of $\Fl$, we have the equalities
		$$
		\overline{\fR_{w/u}(\ttt,\sss)} \;=\; \overline{\fS_u(\ttt,\sss)}\;\overline{\fS_{w_0 w}(\ttt,\sss')} \;=\; \left[X_u\right]^{T'} \left[X^w\right]^{T'} \;=\; \left[R_u^w\right]^{T'}  \;\in\; \HH_{T'}^\bullet\left(\Fl\right)
		$$
		(see \cite[Chapter 10]{AF_Equi_Coho_AG}).
		Thus the double Richardson polynomial describes the $T'$-equivariant class of the Richardson variety.
		For details on Richardson varieties, the reader is referred to, e.g., \cite{Richardson1992,BRION_FLAGS,Speyer2024}.
	\end{remark}

	Our main result in this section is the following theorem.
	
	\begin{theorem}\label{t:equiv_coho_richardson}
		The {\rm(}sign-changed{\rm)} double Richardson polynomial $\fR_{w/u}(\ttt,-\sss)$ is a covolume polynomial. 
	\end{theorem}
	\begin{proof}
		By \autoref{thm_double_Schubert_equiv} and \autoref{l:cohomology_class_opp_sch}, the polynomials $\fS_u(\ttt,\sss)$  and $\fS_{w_0 w}(\ttt,\sss')$ represent the equivariant classes of the irreducible varieties $D_u$ and $D^w$, respectively. 
		Therefore, the result follows from \autoref{cor_product}.
	\end{proof}
	
	\begin{remark}
		Since the product of two covolume polynomials is a covolume polynomial (see \cite[Corollary 2.14]{ALUFFI_COVOL}, \cite[Proposition 4.9]{RSW23}) and covolume polynomials are stable under certain linear transformations (see \cite[Theorem 2.13]{ALUFFI_COVOL}), to show that $\fR_{w/u}(\ttt,-\sss) = \fS_u(\ttt,-\sss)\fS_{w_0 w}(\ttt,-\sss')$ is a covolume polynomial, it suffices to prove that $\fS_w(\ttt,-\sss)$ is a covolume polynomial for all $w \in S_n$.
	\end{remark}                                        
	
	The double Richardson polynomial specializes to many polynomials of interest:
	\begin{enumerate}[\rm (i)]
		\item $\fR_{w_0/u}(\ttt,\sss) = \fS_u(\ttt,\sss)$ is the double Schubert polynomial. 
		\item $\fR_{w_0/u}(\ttt,\mathbf{0}) = \fS_u(\ttt,\mathbf{0}) = \fS_u(\ttt)$ is the ordinary Schubert polynomial.
		\item We say that $\fR_{w/u}(\ttt) = \fR_{w/u}(\ttt,\mathbf{0})$ is the (ordinary) \emph{Richardson polynomial}.
	\end{enumerate}
	
	The next proposition shows that dually Lorentzian polynomials are discretely log-concave.
	
	\begin{proposition}
		\label{prop_disc_log_conc}
		If $h = \sum_\bn a_\bn \ttt^\bn$ is a dually Lorentzian polynomial, we have that 
		$$
		a_\bn^2 \;\,\geq\;\, a_{\bn + \ee_i - \ee_j}a_{\bn - \ee_i + \ee_j}.
		$$
	\end{proposition}
	\begin{proof}
		Since $N\left(\ttt^\bm h(t_1^{-1},\ldots,t_p^{-1})\right) = \sum_{\bn}\frac{a_\bn}{(\bm-\bn)!}\ttt^{\bm-\bn}$ is Lorentzian for some $\bm \in \NN^p$ large enough, by \cite[Proposition 4.4]{LORENTZIAN} we get the inequality in the statement.
	\end{proof}
	
	The following remark states that the truncation of a dually Lorentzian polynomial is dually Lorentzian. 
	This result also follows from \cite[\S 3]{RSW23}, but here we give a short self-contained proof.
	
	\begin{remark}
		\label{prop_dual_Lorentzian_restrict}
		Let $P(\ttt) = \sum_\bn a_\bn\ttt^{\bn}$ be a dually Lorentzian polynomial.
		Let $\bw = (w_1,\ldots,w_p) \in \NN^p$.
		Then the truncation
		$$
		P'(\ttt) \;=\; \sum_{\bn \le \bw} a_\bn \ttt^{\bn}
		$$
		is dually Lorentzian.
	\end{remark}
	\begin{proof}
		Take $\bm \in \NN^p$ large enough such that $\bw \le \bm$ and 
		$$
		Q(\ttt) \;=\; N\left(\ttt^\bm P(t_1^{-1},\ldots,t_p^{-1})\right) \;=\; \sum_{\bn} \frac{a_\bn}{(\bm-\bn)!} \ttt^{\bm-\bn}
		$$ 
		is a Lorentzian polynomial. 
		Then
		$$
		\partial_{t_1}^{m_1-w_1}\cdots \partial_{t_p}^{m_p-w_p}Q(\ttt) \;=\; \sum_{\bn \le \bw} \frac{a_\bn}{(\bw-\bn)!} \ttt^{\bw-\bn}
		$$
		is also Lorentzian, and consequently  
		$$
		\sum_{\bn \le \bw} a_\bn \ttt^{\bn + \bm - \bw} \;=\; \ttt^{\bm - \bw} P'(\ttt)
		$$
		is dually Lorentzian.  
		Finally, \autoref{rem_dual_Lorentzian_monomial_product} implies that $P'(\ttt)$ is dually Lorentzian.
	\end{proof}
	
	We now obtain some consequences for certain polynomials of interest in algebraic combinatorics.
	
	\begin{corollary}
		\label{cor_polynomials_properties}
		The following polynomials have {\rm M}-convex support and are discretely log-concave:
		\begin{enumerate}[\rm (i)]
			\item {\rm(}sign-changed{\rm)} Double Richardson polynomials $\fR_{w/u}(\ttt,-\sss)$.
			\item Richardson polynomials $\fR_{w/u}(\ttt)$.
			\item {\rm(}sign-changed{\rm)} Double Schubert polynomials $\fS_u(\ttt,-\sss)$.
			\item Schubert polynomials $\fS_u(\ttt)$.
			\item Any truncation of the above polynomials. 
		\end{enumerate}
	\end{corollary}
	\begin{proof}
		By \autoref{t:equiv_coho_richardson} and \autoref{prop_dual_Lorentzian_restrict},  these  polynomials and their truncations are dually Lorentzian. 
		
		Since all these polynomials are dually Lorentzian  and Lorentzian polynomials have M-convex support, \cite[§44.6f]{Schrijver} implies that the support of all these polynomials is M-convex (also, see \cite[Corollary 2.12]{ALUFFI_COVOL}).
		The discrete log-concavity statement follows from \autoref{prop_disc_log_conc}.
	\end{proof}          
	
	We now compare the Richardson polynomial $\fR_{w/u}(\ttt)$ with the \emph{skew Schubert polynomial} of Lenart and Sottile \cite{SKEWSCHUBERT} (also, see \cite[Remark 2.16]{WYSER_YONG}).
	
	\begin{remark}
		\label{rem_skew_Schubert}
		Let $w \ge u$ be two permutations in $S_n$.
		Lenart and Sottile \cite{SKEWSCHUBERT} defined the skew Schubert polynomial $\fS_{w/u}$ as the polynomial representative in normal form of the class of the Richardson variety $R_u^w = X_u \cap X^w$ in the cohomology ring $\HH^\bullet(\Fl)$ of the full flag variety $\Fl$.
		The Borel presentation 
		$$
		\HH^\bullet(\Fl) \;\cong\; \frac{\ZZ[t_1,\ldots,t_n]}{\left(e_1(t_1,\ldots,t_n), \ldots, e_n(t_1,\ldots,t_n)\right)}
		$$
		(where $e_i(t_1,\ldots,t_n)$ is the $i$-th elementary symmetric polynomial) yields a $\ZZ$-basis given by the monomials $t_1^{a_1}\cdots t_n^{a_n}$ with $a_i \le n - i$.
		We say that a polynomial representative is in \emph{normal form} if it can be written as a $\ZZ$-linear combination of this distinguished monomial basis.
		Let $n = 4$ and consider the permutations $w = 3412 > 2143 = u$.
		The Richardson polynomial is given by 
		$$
		\fR_{3412/2143}(\ttt) \;=\; \fS_{2143}(\ttt)^2 \;=\; t_{1}^{4}+2\,t_{1}^{3}t_{2}+t_{1}^{2}t_{2}^{2}+2\,t_{1}^{3}t_{3}+2\,t_{1}^{2}t_{2}t_{3}+t_{1}^{2}t_{3}^{2}.
		$$
		However, the skew Schubert polynomial (the polynomial representative in normal form) is given by 
		$$
		\fS_{3412/2143}(\ttt) \;=\; t_{1}^{3}t_{2} \,+\, t_{1}^{3}t_{3} \,+\, t_{1}^{2}t_{2}t_{3}.
		$$
	\end{remark}          
	
	Due to \autoref{cor_polynomials_properties}, if the skew Schubert polynomial $\fS_{w/u}(\ttt)$ is a truncation of the Richardson polynomial $\fR_{w/u}(\ttt)$, then $\fS_{w/u}(\ttt)$ is a dually Lorentzian polynomial. 
	However, as in \autoref{rem_skew_Schubert}, $\fS_{w/u}(\ttt)$ may not coincide with a truncation of $\fR_{w/u}(\ttt)$.
	Thus we should ask the question below.

	\begin{question}
		Let $w \ge u$ be two permutations in $S_n$. 
		We ask the following:
		\begin{enumerate}[\rm (a)]
			\item Is the support of $\fS_{w/u}(\ttt)$ an M-convex set?
			\item Is $\fS_{w/u}(\ttt)$ discretely log-concave?
			\item Is $\fS_{w/u}(\ttt)$ a dually Lorentzian polynomial?
			\item Is $\fS_{w/u}(\ttt)$ a denormalized Lorentzian polynomial?
		\end{enumerate}
	\end{question}
	
	\section{Macaulay dual generators of cohomology rings}
	\label{sect_cohom_rings}
	
	In this section, we study the Macaulay dual generators of cohomology rings of smooth complete irreducible complex algebraic varieties. 
	We prove that under certain positivity assumptions the Macaulay dual generators are denormalized Lorentzian polynomials.

	\subsection{Gorenstein algebras over a base ring}
	\label{subsect_Gor}
	
	Before presenting our results on cohomology rings, which is our main interest, we develop basic ideas regarding the notion of Gorenstein algebras over a base ring. (These developments are probably known to the experts but we could not find a suitable reference for us.)
	For details on Gorenstein rings and duality results, the reader is referred to \cite[Chapter 21]{EISEN_COMM}, \cite[Chapter 3]{BRUNS_HERZOG}.
	Throughout this subsection we use the following setup. 
	
	\begin{setup}
		\label{setup_Gor}
		Let $A$ be a Noetherian ring (always assumed to be commutative and with identity) and $R$ be a positively graded  finitely generated algebra over $R_0 = A$.
		Choose a positively graded polynomial ring $S = A[x_1,\ldots,x_n]$ with a graded surjection $S \surjects R$ and write $R \cong S/I$ for some homogeneous ideal $I \subset S$.
		Let $\mm := (x_1,\ldots,x_n) \subset S$ be the graded irrelevant ideal.
		Let $\delta_i := \deg(x_i) > 0$ and $\delta := \delta_1+\cdots+\delta_n$.
		For any $\pp \in \Spec(A)$, we denote by $\kappa(\pp) := A_\pp/\pp A_\pp$ the residue field at $\pp$ and we set $S(\pp) := S \otimes_A \kappa(\pp) \cong \kappa(\pp)[x_1,\ldots,x_n]$ and $R(\pp):= R \otimes_A \kappa(\pp)$.
	\end{setup} 
	
	We are interested in the following relative notion of Gorenstein.
	
	\begin{definition}
		We say that $R$ is a \emph{Gorenstein algebra over $A$} if  the natural morphism $f\colon \Spec(R) \rightarrow \Spec(A)$ is a Gorenstein morphism (see \cite[\href{https://stacks.math.columbia.edu/tag/0C02}{Tag 0C02}]{stacks-project}).
		This means that $R$ is $A$-flat and the fiber $R(\pp) = R \otimes_A \kappa(\pp)$ is a Gorenstein ring for all $\pp \in \Spec(A)$. 	
		We also say that $R$ is an \emph{Artinian Gorenstein algebra over $A$} if $f\colon \Spec(R) \rightarrow \Spec(A)$ is a finite Gorenstein morphism.
	\end{definition}

	For a graded $S$-module  $M$, we denote the \emph{$A$-relative graded Matlis dual} by
	$$
	M^{\vee_A} \;=\; {}^*\Hom_A(M, A) \;:=\; \bigoplus_{\nu \in \ZZ} \Hom_A\left({M}_{-\nu}, A\right)
	$$ 
	(see \cite[\S 1.5]{BRUNS_HERZOG} for more details on graded duals).
	
	\begin{remark}
		\label{rem_matlis}
		Let $L, M, N$ be finitely generated graded $R$-modules that are $A$-flat. 
		Then: 
		\begin{enumerate}[\rm (i)]
			\item For all $\nu \in \ZZ$, the graded part $M_\nu$ is a finitely generated flat $A$-module, and thus $A$-projective.        
			\item $\left(M^{\vee_A}\right)^{\vee_A} \cong M$.
			\item $\Ann_R\left(M^{\vee_A}\right) = \Ann_R(M)$.
			\item If $0 \rightarrow L \rightarrow M \rightarrow N \rightarrow 0$ is a short exact sequence, then $0 \rightarrow N^{\vee_A} \rightarrow M^{\vee_A} \rightarrow  L^{\vee_A} \rightarrow 0$ is also exact.
		\end{enumerate}
	\end{remark}										
	
	\begin{remark}
		\label{rem_dim_fibers}
		All the fibers $R(\pp) = R \otimes_A \kappa(\pp)$ of $f\colon \Spec(R) \rightarrow \Spec(A)$ have the same dimension in either of the following two cases:
		\begin{enumerate}[\rm (i)]
			\item $f$ is a finite morphism.
			\item $f$ is flat and $\Spec(A)$ is connected (e.g., $A$ is a domain).
		\end{enumerate}
	\end{remark}
	\begin{proof}
		(i) In this case it is clear that all the fibers $R(\pp)$ have Krull dimension $0$.
		
		(ii) Notice that each graded part $R_\nu$ is a locally-free $A$-module of constant rank (see, e.g., \cite[\href{https://stacks.math.columbia.edu/tag/00NX}{Tag 00NX}]{stacks-project}).
		Thus the Hilbert function of the graded $\kappa(\pp)$-algebra $R(\pp)$ is the same for all $\pp \in \Spec(A)$.
	\end{proof}
	
	\begin{remark}
		\label{rem_vanishing}
		Let $M$ be a finitely generated $R$-module. 
		Then $M = 0$ if and only if $M\otimes_A \kappa(\pp) = 0$ for all $\pp \in \Spec(A)$. 
	\end{remark}
	
	\begin{remark}
		A Cohen--Macaulay ring $S$ with a canonical module $\omega_S$ is Gorenstein if and only if $\omega_S$ is generated locally by one element (i.e., $\omega_{S_P} \cong \omega_S \otimes_S S_P$ is generated by one element for all $P \in \Spec(S)$).
	\end{remark}
	
	The next proposition is similar to \cite[Lemma 2.10]{FIB_FULL}.
	
	\begin{proposition}
		\label{prop_gor_char}
		Assume that $f\colon \Spec(R) \rightarrow \Spec(A)$ is flat and that either $\Spec(A)$ is connected or $f$ is finite.
		Let $e$ be the common dimension of the fibers of $f$ {\rm(}see \autoref{rem_dim_fibers}{\rm)}.
		Then $R$ is Gorenstein over $A$ if and only if the following two conditions are satisfied:
		\begin{enumerate}[\rm (a)]
			\item $\Ext_S^{i}(R, S) = 0$ for all $0 \le i \le n$ such that $i \neq n-e$.
			\item $\Ext_S^{n-e}(R, S)$ is $A$-flat and it is generated locally by one element as an $R$-module.
		\end{enumerate}
	\end{proposition}
	\begin{proof}
		Here we use the property of exchange for local Ext modules (see \cite[Theorem A.5]{LONSTED_KLEIMAN}, \cite[Theorem 1.9]{ALTMAN_KLEIMAN_COMPACT_PICARD}).
		For any $\pp \in \Spec(A)$, consider the natural base change maps 
		$$
		\gamma_\pp^i\colon\; \Ext_S^{i}(R, S) \otimes_A \kappa(\pp) \;\rightarrow\; \Ext_{S(\pp)}^{i}(R(\pp), S(\pp)).
		$$
		We have the following two important properties:
		\begin{enumerate}[\rm (i)]
			\item If $\gamma_\pp^i$ is surjective, then it is an isomorphism.  
			\item If $\gamma_\pp^i$ is surjective, then $\gamma_\pp^{i-1}$ is surjective if and only if $\Ext_S^{i}(R, S) \otimes_A A_\pp$ is $A_\pp$-flat.
		\end{enumerate}
		For all $i > n$, since each $\gamma_\pp^i$ is surjective by Hilbert's syzygy theorem, we obtain the vanishing $\Ext_S^{i}(R, S) = 0$ (see \autoref{rem_vanishing}).
		For all $\pp \in \Spec(A)$, recall that the fiber $R(\pp)$ is a Cohen--Macaulay ring if and only if $\Ext_{S(\pp)}^i(R(\pp), S(\pp)) = 0$ for all $i \neq n-e$ (this follows for instance by the local duality theorem).

		Let $i > n-e$.
		If each $\gamma_\pp^{i+1}$ is surjective and $\Ext_S^{i+1}(R, S)=0$, it follows that $\gamma_\pp^{i}$ is surjective, and this implies that $\Ext_S^{i}(R, S) = 0$ if and only if $\Ext_{S(\pp)}^i(R(\pp), S(\pp)) = 0$ for all $\pp \in \Spec(A)$.
		Hence, by descending induction on $i$, we obtain $\Ext_S^{i}(R, S) = 0$ for all $i > n-e$ if and only if  $\Ext_{S(\pp)}^i(R(\pp), S(\pp)) = 0$ for all $i > n-e$ and $\pp \in \Spec(A)$. 
		
		So we assume that $\Ext_S^{i}(R, S) = 0$ for all $i > n-e$ (and, equivalently, that $\Ext_{S(\pp)}^i(R(\pp), S(\pp)) = 0$ for all $i > n-e$ and $\pp \in \Spec(A)$).
		
		Since each $\gamma_\pp^{n-e+1}$ is surjective and $\Ext_S^{n-e+1}(R, S)=0$, it follows that each $\gamma_\pp^{n-e}$ is surjective.
		In turn, this implies that $\gamma_\pp^{n-e-1}$ is surjective for all $\pp \in \Spec(A)$ if and only if $\Ext_S^{n-e}(R, S)$ is $A$-flat.
		Similarly, by descending induction on $i$, we can show that $\Ext_{S(\pp)}^i(R(\pp), S(\pp)) = 0$ for all $i < n-e$ and $\pp \in \Spec(A)$ if and only if $\Ext_S^i(R, S) = 0$ for all $i < n-e$ and $\Ext_S^{n-e}(R, S)$ is $A$-flat.
		
		Therefore, in our current setting, we have shown that $f\colon \Spec(R) \rightarrow \Spec(A)$ is a Cohen--Macaulay morphism (see \cite[\href{https://stacks.math.columbia.edu/tag/045Q}{Tag 045Q}]{stacks-project}) if and only if $\Ext_S^i(R, S) = 0$ for all $0 \le i \le n$ with $i \neq n-e$ and $\Ext_S^{n-e}(R, S)$ is $A$-flat.
		
		Let $M = \Ext_S^{n-e}(R, S)$.
		To conclude the proof, we may assume that $f$ is a Cohen--Macaulay morphism, and we need to show that $M$ can be generated locally by one element if and only if $M \otimes_A \kappa(\pp) \cong \Ext_{S(\pp)}^{n-e}(R(\pp), S(\pp))$ can be generated by one element for all $\pp \in \Spec(A)$.
		This can be proved using basic properties of Fitting ideals (see \cite[\S 20.2]{EISEN_COMM}).
		Indeed, the claim follows by \cite[Proposition 20.6]{EISEN_COMM} because ${\rm Fitt}_1(M) = R$ if and only if ${\rm Fitt}_1(M \otimes_A \kappa(\pp)) = \left({\rm Fitt}_1(M)\right)R(\pp) = R(\pp)$ for all $\pp \in \Spec(A)$ (again, see \autoref{rem_vanishing}).
		So the proof is complete.
	\end{proof}
	
	Motivated by the above proposition, when $R$ is Gorenstein over $A$ and $e$ is the common dimension of the fibers, we say that 
	$$
	\omega_{R/A} \;:=\; \Ext_S^{n-e}\big(R, S(-\delta)\big)
	$$ 
	is a relative canonical module of $R$ over $A$.
	
	Assume that $R$ is a finitely generated $A$-module.
	Let $d > 0$ be a positive integer and $\rho\colon R \rightarrow A(-d)$ be a graded $A$-linear map (following tradition we call this map an \emph{orientation}).
	We say that \emph{$R$ satisfies Poincar\'e duality with respect to the orientation $\rho$} if 
	$$
	R \otimes_A R \;\rightarrow\; A(-d), \quad r_1 \otimes r_2 \mapsto \rho(r_1r_2)
	$$
	is a perfect pairing. 
	This means that the pairing induces a graded isomorphism $R \cong {}^*\Hom_A(R, A)(-d)$ of $R$-modules.
	In particular, we have perfect pairings $R_i \otimes_A R_{d-i} \rightarrow A$ for all $0\le i \le d$.
	We also denote by 
	$$
	\rho\colon R_d \xrightarrow{\;\cong\;} A
	$$ 
	the induced isomorphism.
	Since we are assuming that $R$ is a finitely generated $A$-module, we actually have that ${}^*\Hom_A(R, A)$ coincides with $\Hom_A(R, A)$ (see \cite[Exercise 1.5.19]{BRUNS_HERZOG}), but we prefer the notation ${}^*\Hom_A(R, A)$ to stress our interest in the graded parts of $R$.
	
	\begin{remark}
		\label{rem_hom_tensor_adj}
		By the Hom-tensor adjunction, we have the isomorphism 
		$$
		\Hom_A\left(R, {}^*\Hom_A(R, A(-d))\right) \;\cong\; {}^*\Hom_A\left(R \otimes_A R, A(-d)\right).
		$$
		Thus having a perfect pairing $R \otimes_A R \rightarrow A(-d)$ is the same as having $R \cong {}^*\Hom_A(R, A)(-d)$.
	\end{remark}
	
	We consider the inverse polynomial ring $T = A[y_1,\ldots,y_n]$ where $y_i$ is identified with $x_i^{-1}$.
	The $S$-module structure of $T$ is given by setting that $$
	x_1^{\alpha_1}\cdots x_{n}^{\alpha_n} \,\cdot\, y_1^{\beta_1}\cdots y_{n}^{\beta_n} \;\;=\;\; 
	\begin{cases}
		y_1^{\beta_1-\alpha_1}\cdots y_{n}^{\beta_n-\alpha_n} & \text{if $\beta_i \ge \alpha_i$ for all $1 \le i \le n$} \\
		0 & \text{otherwise.}
	\end{cases}
	$$
	Then $T$ is a negatively graded polynomial ring with $\deg(y_i)=-\delta_i$.
	We have the natural isomorphisms
	$$
	T \;\cong\; {}^*\Hom_A(S, A) \;\cong\; \HL^{n}\big(S\big)(-\delta)
	$$
	of graded $S$-modules. 
	Recall the basic computation $\HL^{n}(S) \cong \frac{1}{x_1\ldots x_n} A[x_1^{-1},\ldots,x_n^{-1}]$.
	
	The following theorem extends well-known results about Artinian Gorenstein algebras (over a field) to our current relative setting over a Noetherian base ring. 										
	
	\begin{theorem}
		\label{thm_rel_Gor}
		Assume \autoref{setup_Gor} and that $f\colon \Spec(R) \rightarrow \Spec(A)$ is a finite flat morphism and that $R$ satisfies Poincar\'e duality with respect to an orientation $\rho\colon R \rightarrow A(-d)$. 
		Then the following statements hold: 
		\begin{enumerate}[\rm (i)]
			\item $R$ is Artinian Gorenstein over $A$.
			\item We have the isomorphisms $\omega_{R/A} = \Ext_S^{n}(R, S(-\delta)) \cong {}^*\Hom_A(R, A) \cong R(d)$.
			\item Consider the inverse polynomial 
			$$
			G_R(y_1,\ldots,y_n) \;=\; \sum_{\alpha_1\delta_1 + \cdots + \alpha_n\delta_n = d} \rho\big(x_1^{\alpha_1}\cdots x_n^{\alpha_n}\big) y_1^{\alpha_1}\cdots y_n^{\alpha_n}  \;\;\in\;\; T = A[y_1,\ldots,y_n].
			$$
			Then the presenting ideal of $R = S/I$ is given as the annihilator of $G_R${\rm:}
			$$
			I \;=\; \lbrace g \in S \mid g \cdot G_R = 0 \rbrace.
			$$			
		\end{enumerate}
	\end{theorem}
	\begin{proof}
		Since $R = \HL^0(R)$ is $A$-flat and $\HL^{i}(R) = 0$ for $i \ge 1$, the generic version of the local duality theorem given in \cite[Theorem A]{FIB_FULL} yields the isomorphisms $\Ext_S^{n-i}(R, S(-\delta)) \cong {}^*\Hom_A(\HL^i(R), A)$ for all $i$.
		Therefore, $R$ is Gorenstein over $A$ by \autoref{prop_gor_char} and we have the isomorphisms $\omega_{R/A} = \Ext_S^{n}(R, S(-\delta)) \cong {}^*\Hom_A(R, A) \cong R(d)$.
		So the proofs of part (i) and (ii) are complete.
		
		Let $E = I^{\perp_A} =\lbrace G \in T \mid g \cdot G = 0 \text{ for all $g \in I$} \rbrace$ be the (relative) inverse system of $I$.
		We have the natural isomorphisms
		$$
		E \;\cong\; \Hom_S\big(R, T\big) \;\cong\; \Hom_S\big(R, {}^*\Hom_A(S, A)\big) \;\cong\; {}^*\Hom_A(R, A) = R^{\vee_A}.
		$$
		The isomorphism $E \cong R^{\vee_A}$ can be obtained by dualizing the short exact sequence $0 \rightarrow I \rightarrow S \rightarrow R \rightarrow 0$ and writing 
		$$
		0 \;\rightarrow\; E \cong R^{\vee_A}  \;\rightarrow\; T \cong S^{\vee_A} \;\rightarrow\; T/E \cong I^{\vee_A} \;\rightarrow\; 0
		$$
		(see \autoref{rem_matlis}).
		By part (ii), there is an inverse polynomial $G \in T = A[y_1,\ldots,y_n]$ of degree $-d$ that generates $E$ as an $R$-module.
		Then the dual short exact sequence $0 \rightarrow I \rightarrow S \rightarrow E^{\vee_A} \rightarrow 0$ shows that 
		$$
		I \;=\; \lbrace g \in S \mid g \cdot G = 0 \rbrace.
		$$
		
		To conclude the proof, it suffices to show that $G_R$ generates the $A$-module $E_{-d}$.
		By construction, for any homogeneous polynomial $g \in S_d$, we have that 
		$$
		g \cdot G_R \;=\; \rho\left(\overline{g}\right),
		$$
		where $\overline{g}$ denotes the class of $g \in S_d$ in $R_d = S_d/I_d$.
		This implies that under the isomorphism $E_{-d} \cong \Hom_A(R_d, A)$, the inverse polynomial $G_R \in E_{-d}$ corresponds to the isomorphism $\rho \in \Hom_A(R_d, A)$.
		On the other hand, under the isomorphism $\Hom_A(R_d,A) \cong A$, the isomorphism $\rho$ should correspond to a unit in $A$. 
		This shows that $G_R$ is a generator of $E_{-d}$, and so the proof of the theorem is complete.
	\end{proof}

	\begin{definition}
		Following standard notation, we say that the inverse polynomial $G_R \in T = A[y_1,\ldots,y_n]$ presented in \autoref{thm_rel_Gor} is the \emph{Macaulay dual generator of $R$ over $A$}.
	\end{definition}

	It is well-known that, over a field $\kk$, an Artinian graded $\kk$-algebra is Gorenstein if and only if it is a Poincar\'e duality algebra (see, e.g., \cite[Theorem 2.79]{Lefschetz_book}).
	The following simple example shows that, over Noetherian base rings, we may not have the converse of \autoref{thm_rel_Gor}.

	\begin{example}    
		\label{examp_Gor_not_Poincare}
		Let $\kk$ be a field and consider the reduced $\kk$-algebra $A = \kk[t]/\left(t(t-1)\right)$.
		Consider the standard graded $A$-algebra 
		$$
		R \;=\; \frac{A[x,y]}{\left(tx,\; (t-1)y,\; x^2,\; xy,\; \left(x,y\right)^3 \right)}.
		$$
		Since we have the natural isomorphism $R \cong A \oplus \frac{A}{(t)}x \oplus \frac{A}{(t-1)}y \oplus \frac{A}{(t-1)}y^2$, it follows that $R$ is $A$-flat.
		Let $\pp_1 = (t) \subset A$ and $\pp_2=(t-1) \subset A$ be the two points of $\Spec(A)$.
		Then the two fibers 
		$$
		R(\pp_1) \;\cong\; \frac{\kk[x]}{\left(x^2\right)} \qquad \text{ and } \qquad R(\pp_2) \;\cong\; \frac{\kk[y]}{\left(y^3\right)}
		$$
		are Artinian Gorenstein algebras over $\kk = \kappa(\pp_1) = \kappa(\pp_2)$.
		Therefore, $R$ is Artinian Gorenstein over $A$. 
		Each of the fibers of $R$ are Poincar\'e duality algebras, but we do not have a \emph{global} Poincar\'e duality of $R$ over $A$ (as we introduced it).
		Indeed, $R$ cannot have a Poincar\'e duality over $A$ because, for instance, the top nonvanishing graded part $R_2 \cong \frac{A}{(t-1)}y^2$ is \emph{not} isomorphic to $A$. 
	\end{example}
	
	The next proposition shows that we get the converse of \autoref{thm_rel_Gor} when the base ring $A$ satisfies a Quillen--Suslin type of theorem (i.e., every finitely generated projective module is free).
	
	\begin{proposition}
		\label{prop_Gor_and_Poincare}
		Assume \autoref{setup_Gor} and one of the following conditions:
		\begin{enumerate}[\rm (a)]
			\item $A$ is a principal ideal domain {\rm(}e.g., $A = \ZZ${\rm)}. 
			\item $A$ is a polynomial ring over a field. 
			\item Or, more generally, every finitely generated projective $A$-module is $A$-free.
		\end{enumerate}
		Then, $R$ is Artinian Gorenstein over $A$ if and only if $R$ is flat and satisfies a Poincar\'e duality over $A$.
	\end{proposition}
	\begin{proof}
		We only need to show the converse of \autoref{thm_rel_Gor}.
		Thus we assume that $R$ is Artinian Gorenstein over $A$ and that every finitely generated projective $A$-module is $A$-free.
		The latter condition is satisfied in the case of PIDs and polynomial rings over a field (see, e.g., \cite{LAM_SERRE_PROB}).
		Let $d$ be the top nonvanishing graded part of $R$.
		Due to our assumptions, $R_d \cong A$ and $R(\pp) = R \otimes_A \kappa(\pp)$ is an Artinian Gorenstein algebra for all $\pp \in \Spec(A)$.
		Consequently, for all $\pp \in \Spec(A)$, the canonical module $\omega_{R(\pp)} = {}^*\Hom_{\kappa(\pp)}\left(R(\pp), \kappa(\pp)\right) \cong {}^*\Hom_A(R, A) \otimes_A \kappa(\pp)$ of $R(\pp)$ is generated by its graded part of degree $-d$.
		This implies that ${}^*\Hom_A(R, A)$ is generated by its graded part of degree $-d$ (see \autoref{rem_vanishing}), and so by choosing a basis generator of the graded part $\Hom_A(R_d, A)$ of degree $-d$, we obtain an isomorphism $\left(R/\aaa\right)(d) \xrightarrow{\cong} {}^*\Hom_A(R, A)$ for some homogeneous ideal $\aaa \subset R$.
		On the other hand, since $\Ann_{R}\left({}^*\Hom_A(R, A)\right) = 0$ (see \autoref{rem_matlis}), we obtain the isomorphism $R(d) \xrightarrow{\cong} {}^*\Hom_A(R, A)$.
		The result of the proposition now follows from \autoref{rem_hom_tensor_adj}.
	\end{proof}
	
	\subsection{Cohomology rings as Artinian Gorenstein algebras over \texorpdfstring{$\ZZ$}{Z}}
	\label{subset_cohom_ring}
	
	In this subsection, we apply our developments in \autoref{subsect_Gor} to cohomology rings.
	Our main result in this direction is the following theorem.
	
	\begin{theorem}
		\label{thm_cohom_Macaulay}
		Let $X$ be a $d$-dimensional smooth complete irreducible complex algebraic variety. 
		Suppose that the even cohomology ring $R = \bigoplus_{i=0}^d\HH^{2i}(X, \ZZ)$ is a flat $\ZZ$-algebra {\rm(}i.e., it is $\ZZ$-torsion-free{\rm)}.
		Let $\rho\colon R_d = \HH^{2d}(X, \ZZ) \rightarrow \ZZ$ be the natural degree map.
		Choose a graded presentation $R \cong  S/I$ where $S = \ZZ[x_1,\ldots,x_n]$, $\delta_i = \deg(x_i) > 0$, and $I \subset S$ is a homogeneous ideal.
		Let $\delta = \delta_1+\cdots+\delta_n$.
		Then the following statements hold: 
		\begin{enumerate}[\rm (i)]
			\item $R$ is Artinian Gorenstein over $\ZZ$.
			\item We have the isomorphisms $\omega_{R/\ZZ} = \Ext_S^n(R, S(-\delta)) \cong {}^*\Hom_\ZZ(R, \ZZ) \cong R(d)$.
			\item Consider the inverse polynomial ring $T = \ZZ[y_1,\ldots,y_n]$ with the identification $y_i = x_i^{-1}$.
			The ideal $I \subset S$ is given as the annihilator
			$$
			I \;=\; \lbrace g \in S \mid g \cdot G_R = 0 \rbrace
			$$			
			of the inverse polynomial 
			$$
			G_R(y_1,\ldots,y_n) \;=\; \sum_{\alpha_1\delta_1 + \cdots + \alpha_n\delta_n = d} \rho\big(x_1^{\alpha_1}\cdots x_n^{\alpha_n}\big) y_1^{\alpha_1}\cdots y_n^{\alpha_n}  \;\;\in\;\; T = \ZZ[y_1,\ldots,y_n].
			$$		
			\item 
			Assume that each $x_i$ is equal to the first Chern class $c_1(L_i)$ of a nef line bundle $L_i$ on $X$.  
			Then the normalization 
			$$
			N(G_R) \;\in\; \RR[y_1,\ldots,y_n]
			$$ 
			of $G_R$ is a Lorentzian polynomial. 
		\end{enumerate}
	\end{theorem}
	\begin{proof}
		Parts (i), (ii), and (iii) follow from \autoref{thm_rel_Gor} because $\HH^\bullet(X, \ZZ)$ satisfies Poincar\'e duality and we are assuming it is $\ZZ$-torsion-free.
		
		We now concentrate on proving part (iv).
		Notice that the normalization 
		$$
		N(G_R) \;=\; \frac{1}{d!} \int \big(x_1y_1 + \cdots + x_n y_n\big)^d \cap [X] \;\;\in\;\; \RR[y_1,\ldots,y_n]
		$$ 
		of $G_R$ is the volume polynomial of $x_1,\ldots,x_n$ divided by $d!$.
		By using Chow's lemma, we can find a proper birational surjective morphism $f\colon X' \rightarrow X$ where $X'$ is a projective variety.
		Let $L_i' = f^*(L_i)$ and $x_i'=c_1(L_i')$.
		The projection formula yields the equality 
		$$
		N(G_R) \;=\; \frac{1}{d!} \int \big(x_1'y_1 + \cdots + x_n' y_n\big)^d \cap [X'] \;\;\in\;\; \RR[y_1,\ldots,y_n].
		$$
		Since each $L_i'$ is nef on $X'$, we have that $N(G_R)$ is a Lorentzian polynomial due to \cite[Theorem 4.6]{LORENTZIAN}. 
	\end{proof}

	We close this subsection with some specific computations. 
	
	\begin{example}[The flag variety $\mathcal{F}\ell_3$]
		By considering the Borel presentation 
		$$
		R \;=\; \HH^\bullet(\mathcal{F}\ell_3, \ZZ) \;=\; \frac{\ZZ[x_1,x_2,x_3]}{\left(x_1+x_2+x_3,\, x_1x_2+x_1x_3+x_2x_3,\, x_1x_2x_3 \right)}
		$$
		of the flag variety $\mathcal{F}\ell_3$, we obtain the Macaulay dual generator 
		$$
		G_R(y_1,y_2,y_3) \;=\; -y_{1}^{2}y_{2}+y_{1}^{2}y_{3}+y_{1}y_{2}^{2}-y_{1}y_{3}^{2}-y_{2}^{2}y_{3}+y_{2}y_{3}^{2}.
		$$
		We know that the elements $\widehat{x}_1=x_1$, $\widehat{x}_2=x_1+x_2$ and $\widehat{x}_3=x_1+x_2+x_3$ correspond to nef line bundles on $\mathcal{F}\ell_3$ (see, e.g., \cite[Chapter 10]{Fulton_Young_Tableaux}).
		By considering the following alternative presentation
		$$
		R \;=\; \HH^\bullet(\mathcal{F}\ell_3, \ZZ) \;=\; \frac{\ZZ[\widehat{x}_1,\widehat{x}_2,\widehat{x}_3]}{\left(\widehat{x}_3,\,\; 
			-\widehat{x}_{1}^{2}+\widehat{x}_{1}\widehat{x}_{2}-\widehat{x}_{2}^{2}+\widehat{x}_{2}\widehat{x}_{3},\,\;
			\widehat{x}_{1}^{2}\widehat{x}_{2}-\widehat{x}_{1}\widehat{x}_{2}^{2}-\widehat{x}_{1}^{2}\widehat{x}_{3}+\widehat{x}_{1}\widehat{x}_{2}\widehat{x}_{3} \right)},
		$$
		we obtain the following Macaulay dual generator
		$$
		G_R(\widehat{y}_1,\widehat{y}_2,\widehat{y}_3) \;=\; \widehat{y}_1^2\widehat{y}_2 + \widehat{y}_1\widehat{y}_2^2
		$$
		whose normalization is a Lorentzian polynomial (as predicted by \autoref{thm_cohom_Macaulay}(iv)).
	\end{example}

	\begin{example}[The Grassmannian $\Gr(2, 4)$]
		The cohomology ring of the Grassmannian $\Gr(2, 4)$ is given by 
		$$
		R \;=\; \HH^{\bullet}(\Gr(2,4), \ZZ) \;=\; \frac{\ZZ[x_1,x_2]}{\left(x_1^3-2x_1x_2,\,\, x_1^2x_2-x_2^2\right)},
		$$
		where $x_i$ is the $i$-th Chern class $c_i(\mathcal{S})$ of the universal subbundle $\mathcal{S}$ on $\Gr(2,4)$.
		Then the Macaulay dual generator is given by 
		$$
		G_R(y_1,y_2) \;=\; 2y_1^4 + y_1^2y_2 + y_2^2.
		$$
	\end{example}

	Finally, we provide a characteristic-free (over $\ZZ$) extension of the celebrated result of Khovanskii and Pukhlikov \cite{KP} showing that the cohomology ring (over $\QQ$) of certain toric varieties can be expressed in terms of differential operators that annihilate the volume polynomial. 
	We follow the notations in \cite{Cox2011}. 
	
	We briefly recall the notion of mixed volume (for more details, the reader is referred to \cite[Chapter IV]{EWALD}, \cite[Chapter 7] {CLO_UsingAG}).
	Let $P_1,\dots,P_n$ be lattice polytopes in $\ZZ^d$ and $\Vol(P_i)$ denote the Euclidean volume of $P_i$, where the unit hypercube has volume $1$. 
	The volume of the Minkowski sum of polytopes $\Vol(y_1P_1+\cdots+y_nP_n)$ is a homogeneous polynomial of degree $d$, and we write
	\[
	\Vol(y_1P_1+\cdots+y_nP_n) \;=\; \sum_{|\alpha| = d} \,\frac{d!}{\alpha!}
	\; \mathrm{MV}_\alpha(P_1,\dots,P_n) \, y_1^{\alpha_1}\cdots y_n^{\alpha_n}
	\]
	where $\mathrm{MV}_\alpha(P_1,\dots,P_n)$ is called the \emph{mixed volume} of $(P_1,\dots,P_n)$ of type $\alpha$.

	Now,  let $X = X_\Sigma$ be a $d$-dimensional smooth complete toric variety and write $R = \bigoplus_{i=0}^d\HH^{2i}(X_\Sigma, \ZZ)$.
	Under the Macaulay inverse system notation of \autoref{thm_cohom_Macaulay}, assume that each $x_i$ is equal to the first Chern class $c_1(L_i)$ of a nef line bundle $L_i$ on $X_\Sigma$.  
	Let $D_i$ be a torus-invariant Cartier divisor on $X_\Sigma$ such that $L_i \cong \OO_{X_\Sigma}(D_i)$ (see \cite[Theorem~4.2.1]{Cox2011}). 
	Let $P_i = P_{D_i}$ be the polytope associated to $D_i$.  
	In this general toric setting, we get the following expression of $G_R$ in terms of mixed volumes. 
	
	\begin{corollary}
		\label{cor_toric_setting}
		Under the notation above, we have
		$$
		G_R(y_1,\ldots,y_n) \;=\; \sum_{\alpha_1 + \cdots + \alpha_n = d} \mathrm{MV}_\alpha(P_1,\ldots,P_n) \, y_1^{\alpha_1}\cdots y_n^{\alpha_n} \;\;\in\;\; T = \ZZ[y_1,\ldots,y_n],
		$$
		where $\mathrm{MV}_\alpha(P_1,\ldots,P_n)$ is the mixed volume of $P_1,\ldots,P_n$ of type $\alpha = (\alpha_1,\ldots,\alpha_n)$.
	\end{corollary}
	\begin{proof}
		From \cite[Theorem 13.4.3]{Cox2011}, we obtain the equality
		\begin{align*}
			\int_{X_\Sigma} \big(x_1y_1 + \cdots + x_n y_n\big)^d  &\;=\; \int_{X_\Sigma} \big(D_1y_1 + \cdots + D_n y_n\big)^d \\
			&\;=\;  \Vol(y_1P_1+\cdots+y_nP_n). 
		\end{align*}
		By comparing coefficients we get $\int_{X_\Sigma} x_1^{\alpha_1}\cdots x_n^{\alpha_n} = \mathrm{MV}_\alpha(P_1,\ldots,P_n)$.
		Therefore, the equality 
		$$
		G_R(y_1,\ldots,y_n) \;=\; \sum_{\alpha_1 + \cdots + \alpha_n = d} \mathrm{MV}_\alpha(P_1,\ldots,P_n) \, y_1^{\alpha_1}\cdots y_n^{\alpha_n}
		$$
		follows from \autoref{thm_cohom_Macaulay}(iii).
	\end{proof}

	We now restrict to the boundary divisors of the toric variety $X=X_\Sigma$.
	Let $n = |\Sigma(1)|$ and $m = \rank\!\!\!\left({\rm Pic}(X_\Sigma)\right)$.
	Let $\rho_1,\dots,\rho_n$ be the one-dimensional rays in $\Sigma(1)$. 
	Assume $D_i$ equals the divisor on $X_\Sigma$ associated to $\rho_i$.
	We have the following isomorphism 
	\[
	\ZZ[x_1,\dots,x_n]/(\mathcal{I}+\mathcal{J})  \;\;\xrightarrow{\;\cong\;}\;\; \HH^\bullet(X_\Sigma,\ZZ), \qquad x_{i}\mapsto D_i,
	\]
	where $\mathcal{I} = \left( x_{i_1}\cdots x_{i_s}\mid i_j \text{ are distinct and } \rho_{i_1}+\cdots+\rho_{i_s} \text{ is not a cone of }\Sigma\right)$ and $\mathcal{J}$ is the ideal generated by linear forms $\sum_{i=1}^n\langle m,u_i\rangle x_i$ for all $m\in M$ in the dual lattice $M$. 
	Moreover, the cohomology ring $\HH^\bullet(X_\Sigma,\ZZ)$ is $\ZZ$-torsion-free.
	This explicit description of $\HH^\bullet(X_\Sigma, \ZZ)$ was obtained by Jurkiewicz \cite{Jurkiewicz1980} and Danilov \cite{Danilov1978}.
	The \emph{volume polynomial} of $X_\Sigma$ is defined as 
	$$
	V(y_1,\dots,y_n) \;:=\; \int_{X_\Sigma} \left(D_1y_1 + \cdots + D_ny_n\right)^{d}.
	$$ 
	Assume that the classes of the divisors $\overline{D}_1,\dots, \overline{D}_m$ on $X_\Sigma$ give a basis of $\Pic(X_\Sigma)$.
	Then the \emph{reduced volume polynomial} of $X_\Sigma$ is given by  
	$$
	\overline{V}(y_1,\dots, y_m) \;:=\; \int_{X_\Sigma} \left(\overline{D}_1y_1 +\cdots+\overline{D}_my_m \right)^{d}.
	$$ 
	Below we have our extension of the aforementioned result of Khovanskii and Pukhlikov \cite{KP}.
	
	\begin{corollary}\label{cor:kp}
		Under the Macaulay inverse system notation of \autoref{thm_cohom_Macaulay}, we obtain the isomorphisms of $\ZZ$-algebras
		\[
		R \;\cong\; \ZZ[x_1,\dots,x_n]/I \;\cong\; \ZZ[x_1,\dots,x_m]/J
		\]
		where 
		$$
		I \;=\; \left\{g\in \ZZ[x_1,\dots,x_n] \,\mid\, g\cdot N^{-1}\left(V(y_1,\dots,y_n)\right)= 0 \right\} \;=\; \mathcal{I}+\mathcal{J}
		$$ 
		and 
		$$
		J \;=\; \left\{g\in \ZZ[x_1,\dots,x_m] \,\mid\, g\cdot N^{-1}\left(\overline{V}(y_1,\dots,y_m)\right) = 0\right\}.
		$$
		Here $N^{-1}$ denotes the inverse of the normalization operator. 
	\end{corollary}
	\begin{proof}
		We concentrate on the isomorphism $R \cong \ZZ[x_1,\ldots,x_n]/I$.
		From \autoref{thm_cohom_Macaulay}(iii), we have $I = \{g \in \ZZ[x_1,\ldots, x_n] \mid g \cdot G_R = 0\}$ and that the Macaulay dual generator is given by 
		$$
		G_R(y_1,\ldots,y_n) \;=\; \sum_{\alpha_1 + \cdots + \alpha_n = d} \left(\int_{X_\Sigma} D_1^{\alpha_1}\cdots D_n^{\alpha_n} \right) \, y_1^{\alpha_1} \cdots y_n^{\alpha_n}.
		$$
		Then the equality $I \;=\; \left\{g\in \ZZ[x_1,\dots,x_n] \,\mid\, g\cdot N^{-1}\left(V(y_1,\dots,y_n)\right)= 0 \right\}$ follows by noticing that 
		$$
		\Vol(y_1,\ldots,y_n) \;=\; d!\, N\left(G_R(y_1,\ldots,y_n)\right).
		$$
		The isomorphism $R \cong \ZZ[x_1,\ldots,x_m]/J$ follows verbatim.
	\end{proof}

	\begin{example}[The Hirzebruch surface $\mathcal{H}_r$]
		Consider the Hirzebruch surface $\mathcal{H}_r$ and label the one dimensional rays in the fan as $\rho_1 = \Cone(-\ee_1+r\ee_2)$, $\rho_2 = \Cone(\ee_2)$, $\rho_3 = \Cone(\ee_1)$, and $\rho_4 = \Cone(-\ee_2)$.
		The cohomology ring of $\mathcal{H}_r$ is given by
		\[
		R \;=\; \HH^\bullet(\mathcal{H}_r,\ZZ) \;\cong\; \frac{\ZZ[x_1,x_2,x_3,x_4]}{( x_1x_3,\, x_2x_4,\, -x_1+x_3,\,  rx_1+x_2-x_4)}.
		\]
		The Macaulay dual generator is given by 
		\[
		G_R(y_1,y_2,y_3,y_4) \;=\; -ry_2^2+ry_4^2+ y_1y_2+y_2y_3+y_3y_4+y_1y_4
		\]
		and the volume polynomial is given by 
		\[
		V(y_1,y_2,y_3,y_4) \;=\; -ry_2^2+ry_4^2+ 2y_1y_2+2y_2y_3+2y_3y_4+2y_1y_4 \;=\; 2!N(G_R(y_1,y_2,y_3,y_4)).
		\]
		Alternatively, the cohomology ring of $\mathcal{H}_r$ can be represented as 
		\[
		R \;=\; \HH^\bullet(\mathcal{H}_r,\ZZ) \;\cong\; \frac{\ZZ[x_3,x_4]}{(x_3^2,-rx_3x_4+x_4^2)}.
		\]
		The Macaulay dual generator is given by 
		\[
		G_R(y_3,y_4) \;=\; ry_4^2+ y_3y_4
		\]
		and the reduced volume polynomial is given by 
		\[
		\overline{V}(y_3,y_4) \;=\; ry_4^2 +2y_3y_4 \;=\; 2!N(G_R(y_3,y_4)).
		\]
		Since $D_{\rho_3}$ and $D_{\rho_4}$ are both nef divisors in $\mathcal{H}_r$ \cite[Example 6.3.23]{Cox2011}, we have that $\overline{V}(y_3,y_4)$ is a Lorentzian polynomial (as predicted by \autoref{thm_cohom_Macaulay}(iv)).
	\end{example}

	\section*{Acknowledgments}
	We thank Matt Larson for asking a question that prompted us to write \autoref{examp_Gor_not_Poincare} and \autoref{prop_Gor_and_Poincare}.
	We thank the reviewers for carefully reading our paper and for several comments and corrections.
	
	\bibliographystyle{amsalpha}
	\bibliography{references}

\end{document}